\documentclass[reqno, a4paper, 11pt]{amsart}

\usepackage[T1]{fontenc}
\usepackage[utf8]{inputenc}
\usepackage[foot]{amsaddr}
\usepackage{amsthm,amsfonts,amssymb,amsxtra,appendix,bookmark,euscript,delarray,dsfont,mathrsfs,stackrel,xifthen,mathtools}
\usepackage[normalem]{ulem}
\usepackage[shortlabels]{enumitem}
\usepackage{xcolor}
\usepackage{latexsym,amssymb}
\usepackage{mathrsfs}
\usepackage{longtable}
\usepackage{lmodern}
\usepackage{microtype}
\usepackage{mathabx}
\usepackage{tikz}
\usetikzlibrary{cd}
\usepackage{adjustbox}
\usepackage[bbgreekl]{mathbbol}
\usepackage{comment}
\usepackage{calc}%
\usepackage{centernot}
\usepackage[includeheadfoot,margin=2.54cm]{geometry}

\usepackage[nodayofweek]{datetime}
\usepackage{tikz-cd}
\usepackage{adjustbox}

\hypersetup{hidelinks}

\DeclareSymbolFontAlphabet{\mathbbm}{bbold}
\DeclareSymbolFontAlphabet{\mathbb}{AMSb}%

\DeclareRobustCommand{\SkipTocEntry}[5]{}

\makeatletter
\newtheorem*{rep@theorem}{\rep@title}
\newcommand{\newreptheorem}[2]{%
	\newenvironment{rep#1}[1]{%
		\def\rep@title{#2 \ref{##1}}%
		\begin{rep@theorem}}%
		{\end{rep@theorem}}}
\makeatother

\theoremstyle{plain}
\newtheorem{theorem}{Theorem}[section]
\newtheorem{lemma}[theorem]{Lemma}
\newreptheorem{lemma}{Lemma}
\newtheorem*{lemma*}{Lemma}

\newtheorem{proposition}[theorem]{Proposition}

\newtheorem*{conjecture*}{Conjecture}
\newtheorem{assumption}[theorem]{Assumption}

\theoremstyle{definition}
\newtheorem{definition}[theorem]{Definition}

\newreptheorem{example}{Example}
\theoremstyle{remark}
\newtheorem{remark}[theorem]{Remark}

\DeclareMathOperator{\Mom}{M}

\mathchardef\mhyphen="2D

\def\bq{\begin{eqnarray}}
	\def\eq{\end{eqnarray}}
\def\bqq{\begin{eqnarray*}}
	\def\eqq{\end{eqnarray*}}

\def\epsilon{\varepsilon}

\newcommand\1{{\ensuremath {\mathds 1} }}

\newcommand{\bsnorm}[2][]{%
	\ifthenelse{\isempty{#1}}%
	{{\ensuremath{|\! |\! |  #2 |\! |\! |_{\beta,s}}}}%
	{{\ensuremath{|\! |\! |  #2 |\! |\! |_{\beta,s}^#1}}}%
}

\def\bN{\mathbb{N}}
\def\bR{\mathbb{R}}

\def\cC{\mathcal{C}}

\def\cF{\mathcal{F}}

\allowdisplaybreaks

\DeclareMathAlphabet{\mathup}{OT1}{\familydefault}{m}{n}
\newcommand{\dx}[1]{\mathop{}\!\mathup{d} #1}

\def\d{\dx}

\DeclarePairedDelimiter{\bra}{(}{)}

\DeclarePairedDelimiter{\set}{\{}{\}}

\makeatletter
\newcommand{\customlabel}[2]{%
   \protected@write \@auxout {}{\string \newlabel {#1}{{#2}{\thepage}{#2}{#1}{}} }%
   \hypertarget{#1}{}
}
\makeatother

\numberwithin{figure}{section}
\numberwithin{equation}{section}

\allowdisplaybreaks

\title[Existence and gelation for CGEDG]{Existence and Non-existence for Continuous Generalized Exchange-Driven Growth Model}

\author{Chun Yin Lam \and André Schlichting}
\email{\{chun.lam,andre.schlichting\}@uni-ulm.de}
\address{Institute for Applied Analysis, University of Ulm, Germany}

\thanks{\emph{Acknowledgement:} The authors thank Prasanta K. Barik for discussions on the generalized EDG model and the contents of~\cite{barik2020mass,barik2024discrete,barik2025continuous}.
The authors thank the anonymous referee for their continuous and detailed reading, whose comments both improved the presentation and corrected inaccuracies in a previous version of the instantaneous gelation result, Theorem~\ref{thm:instant-gel}, leading to the new necessary condition~\eqref{eq:Hcond}.\\[0.5\baselineskip]
\emph{Declaration on the use of AI:}
The proof of the instantaneous gelation Theorem~\ref{thm:instant-gel}, together with the second--difference estimate in Lemma~\ref{lem:secdiff} and Remark~\ref{rmk:Hnecessary} on the necessity of the assumption~\eqref{eq:Hcond} was obtained with the assistance of the large language model Claude (Opus~4.8). The authors verified and revised all statements and proofs and take full responsibility for the contents of this work.
	\\[0.5\baselineskip]
\emph{Funding:} This work is supported by the Deutsche Forschungsgemeinschaft (DFG, German Research Foundation) under Germany's Excellence Strategy EXC 2044 --390685587, Mathematics M\"unster: Dynamics--Geometry--Structure.
The research of AS is partially based upon work from COST Action 24122 mSPACE, supported by COST (European Cooperation in Science and Technology), www.cost.eu.}

\begin{document}
\begin{abstract}
The continuous generalized exchange-driven growth model (CGEDG) is a co\-ag\-u\-la\-tion-fragmentation equation that describes the evolution of the macroscopic cluster size distribution induced by a microscopic dynamic of binary exchanges of masses between clusters. It models droplet formation, migration dynamics, and asset exchanges in various scientific and socio-economic contexts. It can also be viewed as a generalization of the continuous Smoluchowski equations. In this work, we show the existence and uniqueness of solutions for kernels with superlinear growth at infinity and singularity at the origin and show the non-existence of solutions for kernels with sufficiently rapid growth. The latter result is shown via the finite-time gelation and instantaneous gelation in the sense of moment blow-up.
\end{abstract}
\maketitle

\section{Introduction}
The continuous generalized exchange-driven growth model~\eqref{eq:CGEDG} introduced in~\cite{barik2025continuous,lam2025convergence} is a system of integral-differential equations that describes the dynamics of the distribution of cluster masses in a closed system, where masses are exchanged between clusters. We say $c\in C^1([0,T], L^1(\bR_+))$ with $\bR_+:=[0,\infty)$ satisfies the strong form provided that
\begin{equation*}\tag{CGEDG}\label{eq:CGEDG}\begin{split}
\partial_t c (a)& = \int_0^a  \int_z^\infty K(x,a-z,z) c(x)c(a-z)   \d x   \d z  \\
&\quad -
\int_0^a  \int_0^\infty   K(a,x,z) c(a) c(x)  \d x \d z \\
&\quad
- \int_0^\infty \int_z^\infty K(x,a,z)c(x) c(a)  \d x \d z \\
&\quad
+ \int_0^\infty \int_0^\infty K(a+z,x,z) c(x) c(a+z)  \d x  \d z,  \text{ for } a \ge 0,
 \end{split}\end{equation*}
where the kernel $\bR_+^3  \ni (x,y,z) \mapsto K(x,y,z)\ge 0$ is measurable and the time variable is implicit.

By adopting the notation from chemical reaction networks, the system~\eqref{eq:CGEDG} can be seen as the rate equation for the masses $x,y,z\geq 0$ according to the reaction system
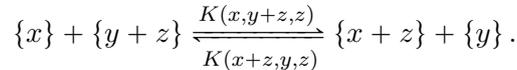
\begin{equation*}
	\set*{x} + \set*{y+z} \xrightleftharpoons[K(x+z,y,z)]{K(x,y+z,z)}
	\set*{x+z} + \set*{y} \,.
\end{equation*}
Here, a cluster of mass $y+z$ exchanges a mass $z$ with a cluster of mass $x$ and the corresponding rate is given by $K(x,y+z,z)$.

The model is also derived as a mean-field limit for a stochastic interacting particle system under an appropriate scaling: Two clusters of discrete particles can exchange an arbitrary number of particles between them with the rate dependent on the masses of the donor and the recipient, as well as the mass being exchanged~\cite{lam2025convergence} (see also~\cite{GrosskinskyJatuviriyapornchai2019,EDG-convergence} for the derivation in the setting of EDG).
In this sense, the system~\eqref{eq:CGEDG} describes the macroscopic dynamics of the distribution of cluster masses with reaction rates prescribed by $K$.

The (non-generalized) exchange-driven growth model (EDG) was first studied in \cite{BenNaimKrapivsky2003} to model physical growth processes with applications in the formation of polymers and droplet formation.
In contrast to EDG, where only a unit mass is exchanged in a reaction, the generalized model might be more suitable for situations with more complicated dynamics, and the restriction of countable sizes is not applicable, for example, in settings of droplet growth and asset exchange~\cite{IspolatovKrapivskyRedner1998}.

The mathematical study of EDG began in  \cite{Esenturk2018}, where fundamental results of well-posedness, local existence and gelation results were discussed. The refinement of the previous results with fast-growing kernels was done in the recent work \cite{si2024existence}.
In~\cite{Schlichting2020,EsenturkValazquez2021}, its long-time behavior was investigated and in~\cite{EichenbergSchlichting2021}, dynamical self-similar solutions for product kernels were investigated.
A first discrete generalization was introduced in~\cite{barik2024discrete}.
In \cite{barik2025continuous,lam2025convergence}, the well-posedness of the generalized model for at most linear growth kernel was derived.

The integral equation~\eqref{eq:CGEDG} is closely related to the continuous Smoluchowski coagulation equation  \cite{Smoluchowski1916}.
The Smoluchowski coagulation and its gelation phenomenon are very well studied using deterministic \cite{banasiak2019analytic,EscobedoMischlerPerthame2002,Fournier2025-iu} and stochastic methods \cite{aldous1999deterministic,jeon1998existence,FournierLaurencot2009} for a large class of kernels, see also~\cite{andreis2023gelation} for generalizations.
The parallel between them can be seen readily from the weak form of the equation~\eqref{eq:swCGEDG} as well as from the stochastic models \cite{norris2000cluster,norris1999smoluchowski, lam2025convergence}.
Moreover, the possibility for gelation is granted by the quadratic structure of the dynamic in the solution.
However, the specific algebraic structure on the test function is different.
Due to the differences in the operator on the test function, the exchange gradient structure requires a different set of algebraic inequalities compared to the Smoluchowski coagulation equation. On the other hand, while both CGEDG and Smoluchowski coagulation-fragmentation equations contain fragmentation terms, the fragmentation in CGEDG is again quadratic in the solution but it is linear for the Smoluchowski coagulation-fragmentation equation.

Furthermore, CGEDG can be viewed as a generalization of the scalar Boltzmann equation \cite{hoare1984quadratic} in which the kernel is symmetric. It is the mean-field equation of the stochastic exchange model, which has applications in modeling heat conduction in materials. The case of bounded kernels was studied in \cite{FERLAND1994, Giroux2008}, and more recently, a class of kernels with at most linear growth in the first two components was studied in \cite{carlen2025spectral}.

The contributions of this work lie in the well-posedness and the gelation phenomenon for CGEDG for a class of symmetric kernel $K$ in the first two components with superlinear growth.
In particular, the well-posedness results improve previous ones in~\cite{barik2025continuous,lam2025convergence} by allowing faster-growing symmetric kernels with singularity at zero.
Finally, the results on gelation encompass finite-time as well as instantaneous gelation, which is detected by the blow-up of the second moment.

\newpage

\subsection{Settings}
\begin{definition}[Weighted Lebesgue spaces]
\[Y_{-\beta,r}  := \{ c \in L^1(\bR_+) : \|c\|_{L^1_{-\beta,r}}:=\|c\|_{-\beta,r}:= \int_0^\infty (x^{-\beta}+x^r) |c(x)| \d x <+\infty  \}  \]
and $Y_{-\beta,r}^+ $ positive cone of $Y_{-\beta,r}$, $r\ge 0$, $\beta \ge 0$.
Moreover, we define
\begin{equation}\label{eq:def:Yinfty}
	Y_{-\beta,\infty} := \bigcap_{n\in \bN} Y_{-\beta,n}
	\quad\text{and}\quad
	Y_{-\beta,\infty}^+ := \bigcap_{n\in \bN} Y_{-\beta,n}^+ \,.
\end{equation}

Similarly,
\[Y_{r}  := \{ c \in L^1(\bR_+) : \|c\|_{L^1_{r}}:=\|c\|_{r}:= \int_0^\infty x^r |c(x)| \d x <+\infty  \}  \] for $r\in\bR$.
\end{definition}
\begin{assumption}[Global assumptions]\label{ass:global}
	Assume
\begin{enumerate}[(i)]\item $K\ge0 $ is symmetric in the first two coordinates, namely, $K(x,y,\cdot)= K(y,x, \cdot)$.
\item $K(x,y,z)=0$ if $z>x$.
\end{enumerate}
\end{assumption}
\begin{remark}
It is a consequence of the symmetry from Assumption \ref{ass:global} that $K(x,y,z)=0$ if $z > y$.
\end{remark}

In the following statements, we will always assume Assumption \ref{ass:global} without explicitly stating it. We state the assumptions for the existence results.
\begin{assumption}[Global existence]\label{ass:existence}
Let $\mu, \nu \in [0,2], \mu+\nu \le 3$ and $\lambda:= \max(\mu,\nu) >1$, $\alpha\ge 0$. Assume
\begin{equation}\label{eq:generalbdd} K(x,y,z) \le \hat{x}^{-\alpha}\hat{y}^{-\alpha} 2^{-1}(\check{x}^\mu  \check{ y}^\nu+ \check{x}^\nu   \check{y}^\mu )  \varphi(z)\end{equation}
 with $\hat{x}=1\wedge x$, $\check{x} = 1 \vee x$,
 $ \varphi \in Y_{-2\alpha,2\lambda}^+$.
For $x \ge 0$, the second derivative satisfies
\begin{equation}\label{ass:2.derivative,K}\partial^2_{1} K(\cdot ,y,z)(x) \le \hat{y}^{-\alpha}\check{y}^\lambda\varphi(z)\end{equation} and
 if $\alpha >0$, there exists a constant $C_\alpha>0$ such that
 \begin{equation}\label{eq:alpha bound K}\widehat{(x-z)}^{-\alpha}K(x,y,z)
 \le C_{\alpha} \hat{x}^{-2\alpha}  \hat{y}^{-\alpha}  \check{x}^{\lambda} \check{y}^{\lambda}   \varphi(z)  \qquad \text{ for } x\ge z  \ge 0 \,. \end{equation}

\end{assumption}
\begin{remark}\label{rmk:example}
\begin{enumerate}
\item Suppose the assumption \eqref{eq:generalbdd} holds, then such a kernel $K$ satisfies~\eqref{eq:alpha bound K}, provided that there exists $\Omega\in(0,1)$ such that
\begin{equation}\label{ass:truncation}\hat{x}^{\alpha} K(x,y,z)\le  \biggl(1-\frac{z}{x} \biggr)^{\alpha}  \hat{y}^{-\alpha} \check{x}^{\lambda}\check{y}^{\lambda} \varphi(z) \quad \text{ for } x-z \le 1 \text{ and } 1-\frac{z}{x}\le \Omega \,, \end{equation}
with $C_{\alpha}:=\Omega^{-\alpha}$.  The justification is given in Proposition \ref{prop: example}.
We observe that both the condition~\eqref{ass:truncation} with $\Omega <1/2$ as well as~\eqref{eq:generalbdd} allow the kernel to have a singularity near zero. Indeed, we can take $\varphi(z)=z^{2\alpha-1+\epsilon}$ near zero so that  $\varphi\in Y_{-2\alpha,2\lambda}(\bR_+)$ with $\epsilon > 0$. Then for $ z = x/2=y/2$, we have $K(x,x/2,x/2)\le 2^{-2\alpha+1-\epsilon} x^{  -1+\epsilon }$.
\item Upon closer examination of the proof, the assumptions above can be slightly relaxed to kernels given as a family of measures $(K(x,y, \d z))_{x \ge 0, y\ge 0}$ with sufficient integrability in $z$ uniformly in $x,y$ derived from \eqref{eq:generalbdd},  \eqref{ass:2.derivative,K} and \eqref{eq:alpha bound K}.  In this case, the second differentiability can be replaced by a bound on the discrete Laplacian $\bigl(\Delta_z (K( \cdot,y,\d z ))\bigr)(x)$.  Then the solution will remain in $L^1$ if the initial data is in $L^1$. This is not surprising because the continuous Smoluchowski coagulation equation in $L^1$ could be interpreted as having an appropriate delta measure in the $z$ component.

\end{enumerate}
\end{remark}
The possible singularity of the kernel at zero and growth at infinity requires a solution space with suitable weighted moments, which are adapted to the kernel.
\begin{definition}[Weak continuity]
A map $[0,T) \ni t\mapsto c_t \in Y_{-\beta, r}^+$ is (weakly) continuous provided that  the map
\[ t\mapsto \int_0^\infty (x^{-\beta}+x^r) f(x) c_t( x) \d x\] is continuous for all $f \in L^\infty(\bR_+)$. It is denoted by $c \in C([0,T),  w\mhyphen Y_{-\beta, r})$.
\end{definition}
With this, the definition of weak solutions to~\eqref{eq:CGEDG} is given as follows.
\begin{definition}\label{def:weaksolution}
Let $T\in (0,\infty]$ and  $c_0\in Y_{-\beta,r}^+$. A weak solution $c$ with initial data $c_0$ is a function $c :[0,T) \to Y_{-\beta, r}^+$ such that
\begin{enumerate}[(a)]
\item $c\in C([0,T), Y_0 )\cap L^{\infty}([0,T), Y_{-\beta,r})$,
\item for all $t\in [0,T)$,
\[\int_0^t \d s \int_0^\infty \d z \int_{0}^\infty \d x \int_0^\infty \d y  \, \kappa[c_s](x,y,z)<+\infty. \]
\item for all $t\in[0,T)$, it holds for all $f\in L^\infty(\bR_+)$  \begin{equation}\label{eq:swCGEDG}\int_0^\infty f(x) [c_t( x) - c_0( x) ]\d x  = \int_0^t \d s \iiint \d z \d x \d y  \, (\Delta_z f)(x) \kappa[c_s](x,y,z) ,
\end{equation} over the integration domain $D:=\{(x,y,z)\in \bR_+^3: x\ge z, y\ge z\}$,
where the discrete Laplacian is given by \begin{equation}\label{eq:discretelaplace}(\Delta_z f)(x):=f(x+z)-2f(x)+f(x-z),\end{equation} and
\[\kappa[c_s](x,y,z):=K(x,y, z) c_s(x) c_s(y).\]
\end{enumerate}
\end{definition}
\begin{remark}
\begin{enumerate}[(i)]
\item  The symmetry and zero extension of $K$ in Assumption~\ref{ass:global} allows to rewrite the strong form~\eqref{eq:CGEDG} as the weak form~\eqref{eq:swCGEDG} by observing that
\begin{equation}\label{eq:wCGEDG}\int_0^\infty f(x) [c_t( x) - c_0( x) ]\d x  = \int_0^t \d s \int_0^\infty \d z \int_z^\infty \d x \int_0^\infty \d y  \, f \cdot \gamma^{x,y,z}  \kappa[c_s](x,y,z) \,,
\end{equation}
where we use the notation
$f\cdot \gamma^{x,y,z}= -f(x)+f(x-z)-f(y) + f(y+z)$.
\item Since for $f_0(x)\equiv 1$, we have that $(\Delta_z f_0)(x) = 0$, the zero moment is preserved along the evolution. Likewise,  for $f_1(x)=x$, we have $(\Delta_z f_1)(x)=0$, however $f_1\notin L^\infty(\bR_+)$ is not admissible in~\eqref{eq:swCGEDG}.
Hence, the first moment is only formally conserved, which is made rigorous under suitable assumptions for the constructed solutions.
\end{enumerate}
\end{remark}
\subsection{Main results}
The main results are well-posedness for kernels with a singularity at zero and a type of gelation results for~\eqref{eq:CGEDG}.
\begin{theorem}[Global Existence]\label{thm:existence} Suppose $K$ satisfies Assumption \ref{ass:existence}. Let $c_0\in Y_{-\alpha, \lambda}^+$; then~\eqref{eq:CGEDG} has a weak solution $c$ in the sense of Definition \ref{def:weaksolution} with $(r,\beta)= (\lambda, \alpha)$.
\end{theorem}
\begin{remark}
The proof of existence is based on an argument for $L^1$ compactness for a suitable truncated system with ideas and methods from related works for the exchange-driven growth and the Smoluchowski coagulation equation.
We are able to derive the required estimates for~\eqref{eq:CGEDG} under suitable assumptions on the kernel to apply an Arzel\'{a}-Ascoli argument to obtain a subsequent limit.
The limit is then shown to solve~\eqref{eq:CGEDG} in the weak sense.
With the structure of discrete Laplacian \eqref{eq:discretelaplace}, we can adapt the methods \cite{si2024existence} applied to the exchange-driven growth model and translate the techniques to its continuous variant.
Together, we are able to show the well-posedness for kernel growth at infinity up to degree $3$ in the sum of the powers of $x,y$, given sufficient decay in the $z$ component in the kernel.
In addition, for the singularity near zero, we take inspiration from the existence results for the Smoluchowski coagulation equation with singular kernel \cite{barik2020mass}.
Similar to the works~\cite{barik2020mass}, we use a by-now standard argument to first show a compactness in a weak $L^1$ topology and then improve the convergence with a suitable moment estimate.
\end{remark}
	The next result is uniqueness for the constructed weak solutions to~\eqref{eq:CGEDG} by an adaptation of \cite[Theorem 6]{Esenturk2018} in the setting of the exchange-driven growth model. The proof is independent of the existence proof based solely on the weak formulation from Definition~\ref{def:weaksolution}.  We point out that the setting of the uniqueness theorem is more general than our existence result in Theorem~\ref{thm:existence}, as such the existence of the solution on the larger  $Y_{-2\alpha,2\lambda}^+$ space is not proved.

\begin{theorem}[Uniqueness]\label{thm:uniquness} Let $c_0 \in Y_{-2\alpha,2\lambda}^+$, $\varphi \in Y_{-\alpha,\lambda}^+$, $K(x,y,z)  \le \check{x}^{-\alpha}\check{y}^{-\alpha} \check{x}^\lambda  \check{y}^\lambda \varphi (z)$, $\lambda \in [1,2]$ and if $\alpha > 0$ assume in addition \eqref{eq:alpha bound K} holds, then the weak solution to~\eqref{eq:CGEDG} on $[0,T)$, $T \in (0,\infty]$,  is unique in $Y_{-2\alpha,2\lambda}^+$.
\end{theorem}

In the discussion of the gelation and finite-time existence, we consider kernels with growth at infinity but not at the origin. These assumptions and methods are adaptations of the results of~\cite{si2024existence} to the continuous setting. The gelation here is interpreted as the blow-up of the second moment.

\begin{definition}[Weak gelation] Let $c$ be a weak solution to~\eqref{eq:CGEDG}. The (weak) gelation time is defined as
\[T_{gel}:=\sup\{ t\ge 0: \Mom_2 (c_t) <+\infty \}.  \]
Hereby, for $r>0$ the $r$-moment is defined by $\Mom_{r}(c) = \int x^r c(x)\d x$.
\end{definition}

\begin{remark}
\begin{itemize}
\item For the Smoluchowski equation, it is shown that in \cite[Lemma 9.2.2]{banasiak2019analyticii} and see also the discussions in \cite[Section 5.1]{Hendriks1983}, the gelation in the sense of non-conservation of first moment, defined with $\hat{ T}_{ gel}:=\inf\set*{t\ge 0: \Mom_1(c(t)) < \Mom_1(c(0))}$, is equivalent to the blow-up of some higher moment. Since the boundedness of the second moment implies the conservation of the first moment, the blow-up of the second moment is a weaker notion of gelation. It is not yet clear whether the same equivalence holds for the (CG)EDG model in general.
\item Heuristically, since $(x-z)^2 + (y+z)^2 \ge x^2 + y^2 $ if and only if $y+z \ge x$, the growth of the second moment detects the formation of large clusters also for EDG and we refer to~\cite{aldous1999deterministic} for an in-depth discussion of the phenomenon.
\end{itemize}
\end{remark}
\begin{theorem}[Finite-time existence for quadratic growth]\label{thm:finite-existence}
Assume $K$ satisfies $K(x,y,z)\le \check{x}^2 \check{y}^2 \varphi(z)$ and the bound~\eqref{ass:2.derivative,K} from Assumption~\ref{ass:existence} with $\alpha =0, \lambda =2$. Moreover, let $\varphi \in Y_{0,2}^+$,
then for any $0 \nequiv c_0 \in Y^+_{0,2}$ the weak solution to~\eqref{eq:CGEDG} in $ Y_{0,2}^+$  on
$[0,T_0)$, $T_0:=\bra*{ 2 \|\varphi\|_{0,2}\bigl( \Mom_0(c_0)+\Mom_2(c_0)\bigr)}^{-1}$ exists. Moreover, it preserves the first moment.
\end{theorem}

\begin{theorem}[Finite-time gelation]\label{thm:blow-up-time} Assume $K$ satisfies  $\varphi_1(z)(\check{x}^2 \check{y}^\mu + \check{x}^{\mu} \check{y}^2)\1_{(z,+\infty)^2}(x,y)\le K(x,y,z)$ and $K(x,y,z)\le \check{x}^2 \check{y}^2 \varphi(z)$, for $\mu \in (1,2]$, the Equation~\eqref{ass:2.derivative,K} from Assumption~\ref{ass:existence} with $\alpha =0, \lambda =2$ and $\varphi\in Y^+_{0,2}$ and $\varphi_1 \in Y^+_{0,1+\mu}$.
Suppose $0 \nequiv c_0 \in Y_{0,1+\mu}^+$ satisfies $\Mom_2(\varphi_1) \Mom_{\mu}(c_0)>\Mom_{1+\mu}(\varphi_1)\Mom_1(c_0)$. Then it holds
\[\Mom_{\mu}(c_t) \ge \bra*{\frac{1}{\Mom_{\mu}(c_0)-\frac{\Mom_{1+\mu}(\varphi_1)}{\Mom_2(\varphi_1)}\Mom_1(c_0)} - \mu(\mu-1)2^{\mu-2}\Mom_2(\varphi_1) \, t }^{-1}+\frac{\Mom_{1+\mu}(\varphi_1)}{\Mom_2(\varphi_1)}\Mom_1(c_0).\]
Moreover, the gelation time of the weak solution as constructed in Theorem \ref{thm:finite-existence} is finite and  satisfies
\[T_{gel} \le \bra*{\mu(\mu-1)2^{\mu-2}\Mom_2(\varphi_1)  \bra*{\Mom_{\mu}(c_0)-\frac{\Mom_{1+\mu}(\varphi_1)}{\Mom_2(\varphi_1)}\Mom_1(c_0)}}^{-1}.\]
\end{theorem}
\begin{remark}
	In the bound for $T_{gel}$, the integrability of $\varphi$ is the crucial addition in comparison to the results for the discrete EDG model from~\cite{si2024existence}.
	The key arguments in gelation are to derive a moment bound of the solution in the existence time interval.

	In particular, under the assumptions of Theorem~\ref{thm:blow-up-time}, there is no global  mass conserving weak solution  $c$ to~\eqref{eq:CGEDG} in $ Y_{0,2}^+$. Indeed, suppose there exists a global mass conserving solution in $ Y_{0,2}^+$, then for $\mu\in (1,2]$, it holds $\Mom_{\mu}(c(t))<+\infty$ for all $t\ge 0$, which contradicts the finite blow up of $\Mom_\mu(c(t))$ from Theorem \ref{thm:blow-up-time}.
\end{remark}
\begin{theorem}[Instantaneous gelation]\label{thm:instant-gel}
Assume there exists $\beta >2$ such that $K$ satisfies $\varphi_{1}(z) (\check{x}^{\beta}+\check{y}^{\beta})\1_{(z,+\infty)^2}(x,y) \le  K(x,y,z)$ and $K(x,y,z) \le\varphi(z)( \check{x}^{k} + \check{y}^k )$ for some $k\in\bN$ with $k> \beta$ and $\varphi,\varphi_1\in Y_{0,\infty}^+$,
where the rate function $\varphi_1$ charges every neighbourhood of the origin, that is\footnote{Equivalently, $0$ lies in the support of $\varphi_1$.}
\begin{equation}\label{eq:Hcond}\tag{H}
\int_0^\delta \varphi_1(z)\,\d z>0\qquad\text{for every }\delta>0 \,.
\end{equation}
Then for any weak solution of~\eqref{eq:CGEDG} $(c_t)_{t\ge 0}$ in $ Y_{0,2}^+$ with initial datum $c_0 \in  Y_{0,\infty}^+$ and $\Mom_0(c_0) >0$, instantaneous gelation occurs, i.e.  $T_{gel} =0$.
\end{theorem}
\begin{remark}\label{rmk:Hnecessary}
The assumption~\eqref{eq:Hcond} is necessary; it means that $0$ lies in the support of $\varphi_1$, i.e.\ $\varphi_1$ charges every neighbourhood of the origin.
Indeed, if $\varphi_1$ vanishes on $[0,z_0)$ for some $z_0>0$, then~\eqref{eq:Hcond} fails; only exchanges of mass $z\ge z_0$ are then possible, so for an initial datum supported in $(0,z_0)$ no exchange can occur, the solution remains equal to $c_0$, its second moment stays bounded by $z_0^2\Mom_0(c_0)$, and instantaneous gelation fails.
\end{remark}

\begin{remark} In comparison with the statement in the discrete setting \cite[Theorem 2.9]{si2024existence}, the upper bound on the kernel is needed to admit a wide class of functions satisfying the weak form. For the instantaneous gelation for the continuous Smoluchowski equation, a corresponding upper bound in \cite[Volume 2, Theorem 9.2.1]{banasiak2019analyticii} is assumed.

The proof strategy for the instantaneous gelation, showing finiteness of all moments in the existence interval and deriving an upper bound for the blow-up time via higher moments, which converges to zero asymptotically, appeared in~\cite{van1987instantaneous, carr1992instantaneous} for the Smoluchowski equation.

The result also shows that for $K$ satisfying the assumptions of Theorem~\ref{thm:instant-gel} and $c_0 \in  Y_{0,n}^+$ for all $n\in\bN$, there is no weak solution $(c_t)_{t \ge 0}$ to~\eqref{eq:CGEDG} in $Y_{0,2}^+$ on any interval $[0,T)$ for $T>0$.
Indeed, if such a solution exists for some $T>0$, then from the propagation of lower moments (proven in Lemma~\ref{lem:propagation of moment for lower bdd kernel} below), we get $\Mom_p(c_t)<+\infty$ for all $t\in[0,T)$ and any $p \in \bN$ which contradicts $T_{gel}=0$ from Theorem \ref{thm:instant-gel}.
\end{remark}
\subsection{Open questions}
In this work, we used the $L^1$ framework for the solution. The assumptions on kernel \eqref{ass:2.derivative,K} and \eqref{eq:alpha bound K}  were needed to ensure uniform integrability.  In particular, we need $K(x,y,z)$ to be small as $z$ approaches $x$. In this framework, the formation of atoms is not allowed. However, it would also be reasonable to consider measure-valued solutions, as has been done for the Smoluchowski coagulation equations.
This would enable a unified framework for the discrete and continuous models.  We also note that the Smoluchowski coagulation equations are more well-studied than the full coagulation-fragmentation equations. The similarity to the Smoluchowski coagulation equations and the symmetry of exchange dynamics imply that while it is possible to use a similar strategy as the coagulation equations for~\eqref{eq:CGEDG}, one can treat both coagulation and fragmentation effects simultaneously. For the Smoluchowski coagulation equations, the measure-valued solutions were studied in \cite{norris2014measure, Fournier2006} with more recent works on the multi-component generalizations \cite{ferreira2025measure}.

A related question is the shattering phenomenon, that is, the formation of atomic mass (e.g. at zero) in the solution $c$ from a diffuse initial condition. This is analogous to the shattering phenomenon in the Smoluchowski (coagulation-fragmentation) equation. In the case of Smoluchowski equations, it would also lead to the non-existence of solutions. Nevertheless, due to the differences in the fragmentation terms, new methods would be required. In addition, as we observe in this work, one needs different estimates for small cluster sizes $(x,y\ll 1)$ and large cluster sizes $(x,y\gg 1)$ for singular kernels. Intuitively, the competition of the singularity at zero and growth at infinity in the kernel leads to strong interaction between small and large clusters. Its effects on the phase transition remain open.

\section{Existence from the convergence of truncated system}\label{sec:existence}
The proof of Theorem \ref{thm:existence} uses the by-now classical technique of weak $L^1$ compactness, which has been successfully used for EDG and other related coagulation-fragmentation equations. For this reason, we introduce the truncated system and consider its compactness.
\begin{definition}[Symmetric truncated kernel]\label{def:truncateKernel}
The truncated kernel on $(1/n,n)$, $2 \le n\in \bN$, is defined by, for $x,y,z\in\bR_+$
\[K_n(x,y,z)=K(x,y,z) \1_{(1/n,n)^3}(x,y,z) \1_{(0,n)^2}(x+z,y+z)\1_{(1/n,\infty)^2}(x-z,y-z).\]
\end{definition}
Based on the truncated kernel $K_n$ from Definition~\ref{def:truncateKernel}, we arrive at the truncated equation, which is given for $x \ge 0$ by
\begin{equation}\label{eq:truncated} \begin{split}\partial_t c^n_t (x)= &\iint \d z  \d y \, \kappa_n[c_t^n](y,x-z,z) - \iint \d z \d y\, \kappa_n[c_t^n](x,y,z)   \\
&- \iint \d z \d y \, \kappa_n[c_t^n](y,x,z) + \iint \d z \d y \, \kappa_n[c_t^n](x+z, y,z ) \,,
\end{split}\end{equation}
where now $\kappa_n[c^n](x,y,z):= K_n(x,y, z) c^n(x) c^n(y)$.
Likewise, a given initial datum $c_0 \in L^1(\bR_+)$ gives rise to initial data of the truncated system by the truncation $c^n_0(x)= c_0(x) \1_{(1/n,n)}(x)$. The above truncation on $K$ and $c$ ensures $c_t^n(x) = c_0^n(x)=0$ for $x \le 1/n $ or $x \ge n$.
\begin{lemma}\label{lem:weak-form}
Let $n \ge 2$. If $(c^n_t)_{t\ge 0}$ is a classical solution of the truncated system  \eqref{eq:truncated} on $[0,\infty)$, then for $f\in L^\infty((1/n,n))$ it holds
 for all $t\ge 0$
\begin{equation}\label{eq:weak-truncated}\begin{split}
\int_{1/n}^n  f(x) (c^n_t( x) - c^n_0( x)) \d x&
=\int_0^t \d s \iiint\d z \d x \d y \, (\Delta_z f)(x) \kappa_n[c^n_s](x,y,z).
\end{split}
\end{equation}
\end{lemma}
\begin{proof}
	The Lemma is an immediate consequence of the symmetry of the kernel based on Assumption~\ref{ass:global} and a change of variable in $x$.
	See also \cite[Section 3.1]{barik2025continuous} for a similar calculation.
\end{proof}
\begin{proposition}\label{prop:truncated-sol-moment-preserve} Suppose $K$ satisfies $K_n(x,y,z)\le f(n) \varphi(z)$ for $f :\bN \to \bR_+$, $\varphi \in L^1(\bR_+)$.
Then for every $n \ge 2$, the truncated system \eqref{eq:truncated} has a unique non-negative solution $c^n \in C^1([0,\infty), L^1((1/n,n)))$. Furthermore, for any $t\ge 0$, it conserves the mass
\begin{equation}\label{eq:cons-mass}\int^n_{1/n} c_t^n(x) \d x = \int_{1/n}^n  c_0^n(x) \d x,\end{equation} and the first moment
\begin{equation}\label{eq:cons-first-mom}\int^n_{1/n} x c_t^n(x) \d x = \int_{1/n}^n x c_0^n(x) \d x \,.
\end{equation}
\end{proposition}
\begin{proof}
	The result can be proven via the Picard-Lindelöf theorem, as done
	in~\cite[Proposition 3.2]{barik2025continuous} and
	in~\cite[Proposition 4.1]{barik2020mass}. For completeness, we provide a proof in our settings here.
	Using the assumption of $K$ and $\varphi$, we have
\[K_n(x,y,z) \le f(n) \varphi(z) \text{ for } n \ge 1.\]
We show that the right side of \eqref{eq:truncated} is locally Lipschitz in $L^1(1/n,n)$.  Let $c, c' \in L^1(1/n,n)$. We consider the norm in $L^1(1/n,n)$ of each of the terms in the right side \eqref{eq:truncated}. With a change of variable, we get
\begin{equation}\begin{split}
 \int & \d x \,\biggl| \iint \d z  \d y \, \bigl( \kappa_n[c](y,x-z,z) -  \kappa_n[c'](y,x-z,z) \bigr) \biggr| \\
 &\le  \int  \d x  \iint \d z  \d y \, K_n(y,x-z,z) \bigl|c(y)c(x-z) - c'(y) c'(x-z)\bigr|\\
 &\le  \int \d x  \iint \d z  \d y \, K_n(y,x,z) \bigl|c(y)c(x) - c'(y) c'(x)\bigr|\\
 &\le f(n) \| \varphi\|_{0 }( \| c\|_{L^1(1/n,n)} + \| c'\|_{L^1(1/n,n)} )  \| c-c'\|_{L^1(1/n,n)}.
\end{split}
\end{equation} So that $c \mapsto \iint \d z  \d y \, \kappa_n[c_t^n](y,\cdot-z,z) $ is a local Lipschitz map in $L^1(1/n,n)$. Similar calculations for each of the four terms imply the right side induces a locally Lipschitz function on $L^1(1/n,n)$. Therefore by the Picard-Lindelöf theorem there exists a unique solution of the initial value problem $c^n \in C^1([0,\mathcal{T}), L^1((1/n,n)))$ up to a maximal time $\mathcal{T}\in (0,\infty]$ and has blow-up in the sense that $\lim_{t\to \mathcal{T}}\|c^n_t\|_{L^1(1/n,n)}=+\infty$ if $\mathcal{T}<+\infty$.

For the positivity of $c^n$, we note that the positive part of a locally Lipschitz function is also locally Lipschitz. Therefore, the Picard-Lindelöf theorem implies the existence and uniqueness of solutions	 of the initial value problem \begin{equation}\label{eq:postruncated}
\begin{split}\partial_t c^n_t (x)= & \biggl(\iint \d z  \d y \, \kappa_n[c_t^n](y,x-z,z) \biggr)_{\!+}- \iint \d z \d y\,  \kappa_n[c_t^n](x,y,z)   \\
&- \iint \d z \d y \, \kappa_n[c_t^n](y,x,z) + \iint \d z \d y \, \kappa_n[c_t^n](x+z, y,z ) \text{ for } x \ge 0 \,,
\end{split}\end{equation}
where for $a\in \bR$ the notation $(a)_+ = \max\set*{0,a}$ denotes the positive part.
We will now show that $c^n_t  \ge 0$ for $t\in [0,\mathcal{T})$. For this, we calculate $\frac{\d}{\d t}(-c^n_t)_+  =\operatorname{sign}_+(-c^n_t) \frac{\d }{\d t} (-c^n_t)$, where $\operatorname{sign}_+(x) =1 $ for $x \ge 0$ and $0$ for $x <0$, so that
\[\begin{split} \frac{\d}{\d t} & \|(- c^n_t)_+  \|_{L^1(1/n,n)} = - \int_{1/n}^n   \operatorname{sign}_+(-c^n_t) \frac{\d}{\d t} c^n_t(x) \d x \\
& \le  \int  \d x \,\operatorname{sign}_+(-c^n_t)  \\
&\qquad \bra*{ \iint \d z \d y\, \kappa_n[c_t^n](x,y,z)   + \iint \d z \d y \, \kappa_n[c_t^n](y,x,z) +(-1) \iint \d z \d y \, \kappa_n[c_t^n](x+z, y,z )}. \end{split} \]
Using the bound on $K_n$ with a change of variable from $x+z \to x$ in the last integral, we can bound each of the three integrals with $ f(n)\| \varphi\|_{0} \| (-c^n_t)_+\|_{L^1(1/n,n)} \|c^n_t\|_{L^1(1/n,n)} $. Therefore, we have the differential inequality,
\[ \frac{\d}{\d t}  \|(- c^n_t)_+  \|_{L^1(1/n,n)} \le  3 f(n)\| \varphi\|_{0} \| (-c^n_t)_+\|_{L^1(1/n,n)} \|c^n_t\|_{L^1(1/n,n)} \] and Grönwall's lemma implies
\[  \|(- c^n_t)_+  \|_{L^1(1/n,n)} \le \|(- c^n_0)_+  \|_{L^1(1/n,n)} \exp\biggl(  3 f(n)\| \varphi\|_{0}  \int_0^t \|c^n_s\|_{L^1(1/n,n)} \d s \biggr).\]
With the non-negativity of the initial condition, we conclude \[  \|(- c^n_t)_+  \|_{L^1(1/n,n)} \le 0\] so that $c^n_t \ge 0$ for $t\in [0,\mathcal{T})$ and hence the equation~\eqref{eq:postruncated} agrees with~\eqref{eq:truncated}.

By the rewriting of Lemma \ref{lem:weak-form} and the fact that $\{1, x\mapsto x\}$ are in the kernel of $\Delta_z$, the conservation of the zeroth~\eqref{eq:cons-mass} and first moment~\eqref{eq:cons-first-mom} follow.
Moreover, recall that the truncation introduced implies $c_t^n(x) = c_0^n(x)$ for $x \le 1/n $ or $x \ge n$.  The conservation of mass and the non-negativity of $c^n_t$ imply
\[
	\|c^n_t\|_{L^1(1/n,n)} = \|c^n_0\|_{L^1(1/n,n)} \qquad \forall t \in [0,\mathcal{T}) \,.
\]
In particular, blow-up does not occur in $L^1$ so that $\mathcal{T}=+\infty$.
\end{proof}
\begin{remark}
The assumption of Proposition~\ref{prop:truncated-sol-moment-preserve} holds under the global existence Assumptions~\ref{ass:existence} and the assumptions of the local existence Theorem~\ref{thm:finite-existence}. In the latter theorem, we apply the arguments for global existence in this remaining part of this section, modulo the fact that the estimates below can only hold up to some finite time.
\end{remark}
We will extend $(c^n_t)_{t\ge 0}$ to $\bR_+$ by setting $c^n_t(x)=0$ for $x \ge n$ or $x \le 1/n$.
\begin{definition}[Mixed moments]
	For $\alpha \ge 0$, $\lambda > 0$ the mixed $(-\alpha,\lambda)$ moment is defined by
\[\Mom_{-\alpha,\lambda}(c) :=  \int_0^\infty  \, \hat{y}^{-\alpha}   \check{y}^{\lambda}  c(y) \d y \,.\]
\end{definition}

\begin{lemma}[Propagation of mixed moments]\label{lem:prop-lambda-moment}
Let $T\in (0,\infty)$ and let $c_0 \in Y_{-\alpha,\lambda}^+$.
Then there exists $C>0$ depending only on the constants in Assumption~\ref{ass:existence} and $c_0$ such that
\[\Mom_{-\alpha,\lambda}(c^n_t) \le   C T \exp(C T)\quad \forall t \in [0,T] \quad \forall n>1.  \]
Moreover, the family $\set*{(c_t^n)_{t\in [0,T]}}_{n\in \bN}$ is $Y_{-\alpha,0}$-equicontinuous in time, that is there exists $C>0$ independent of $n\in \bN$ such that
for $ 0 \le s \le t \le T$ it holds
\begin{equation}\label{eq:equicontinuity}
	\int_{1/n}^n (1+x^{-\alpha}) \bigl|c^n_t(x)-c^n_s(x)\bigr| \d x  \le C (t-s) .
\end{equation}
\end{lemma}
\begin{proof}
We use $h(x)=\hat{x}^{-\alpha} \check{x}^{\lambda}$ for $x>0$ as test-function in the weak truncated form~\eqref{eq:weak-truncated} and get
\begin{equation*}
\int_{1/n}^n  \, \hat{x}^{-\alpha}  \check{x}^{\lambda} (c_t^n (x)- c^n_0(x))  \d x=  \int_0^t \d s \iiint\d z \d x \d y\, (\Delta_z h)(x)  \kappa_n [c^n_s](x,y, z) .
\end{equation*}
From its definition~\eqref{eq:discretelaplace}, the discrete Laplacian of $h$ splits up into three mutually exclusive cases, which are
\begin{alignat*}{2}
	(\Delta_z h)(x) &= (\Delta_z p_{\lambda})(x) &&\text{ if } x-z \ge 1,  \\
	(\Delta_z h)(x) &=(\Delta_z p_{-\alpha})(x) &&\text{ if } x+z \le 1, 	\\
	(\Delta_z h)(x) &\le (1+2z)^{\lambda}+ (x-z)^{-\alpha} \quad &&\text{ if } 1-z \le x\le 1+z,
\end{alignat*} where we used $p_r(x):=x^r$ for $r \in \bR$.
Hence, we arrive at the splitting
\begin{multline}\label{eq:DeltaSplit}
\iiint \d z \d x \d y\, (\Delta_z h)(x)  \kappa_n[c^n_s](x,y, z)  \\
=\iiint  \d z \d x \d y\, \biggl[\1_{[1+z,\infty)}(x)  (\Delta_z p_\lambda)(x) + \1_{[z,1-z]}(x) \1_{[0,1/2]}(z) (\Delta_z p_{-\alpha})(x) \\
 +\1_{[z \vee (1-z),1+z]}(x) (\Delta_z h)(x)\biggr] \kappa_n[c^n_s](x,y, z).
\end{multline}
We now estimate the first integral in~\eqref{eq:DeltaSplit}. Since the support of $\kappa_n[c^n_s]$ is contained in $\{(x,y,z): x\ge z\}$, we have the following cases: If $x/2 \le z \le x$, then \[(\Delta_z p_{\lambda})(x) \le p_{\lambda}(x+z) \le p_{\lambda}(3 z) = (3 z)^\lambda   \] as $p_\lambda(x)=x^{\lambda}$ is an increasing function.
Otherwise we have $0\le z \le x/2$, then \begin{align}(\Delta_z p_{\lambda})(x) \le & (p_{\lambda}'(x+z) - p_{\lambda}'(x-z)) z\le 2 p_{\lambda}''(x-z) z^2 \nonumber \\
= & 2 \lambda(\lambda -1) \Bigl(\frac{x-z}{x}\Bigr)^{\lambda -2 } x^{\lambda -2 } z^2 \le \lambda(\lambda -1) 2^{3- \lambda}  x^{\lambda -2 } z^2,  \end{align}
since $p_{\lambda}''(x)= \lambda( \lambda-1) x^{\lambda-2 }$ is non-increasing for $x\ge 0$.
With these preliminary bounds, we can now estimate the first integral in~\eqref{eq:DeltaSplit} using also Assumption~\ref{ass:existence}. Indeed, we get
\[\begin{split}
\MoveEqLeft \iiint \d z \d x \d y\, \1_{[1+z,\infty)}(x)  (\Delta_z p_\lambda)(x) \kappa_n[c^n_s](x,y,z) \\
&\le \iiint \d z \d x \d y\, \1_{[1+z,\infty)}(x)\bra*{ \! (3z)^\lambda   \1_{[z,2z]}(x)+  \lambda(\lambda -1) 2^{2- \lambda} z^2    x^{\lambda -2 } \1_{[2z, \infty)}(x)\! }\kappa_n[c^n_s](x,y,z)\\
&\le   \int_{\mathrlap{\bR_+}}\mkern4mu \d z  \, (3z)^{\lambda}2^{\lambda}\check{z}^{\lambda} \varphi(z)\int_{\mathrlap{\bR_+}}\mkern4mu \d x \, c^n_s(x) \int_{\mathrlap{\bR_+}}\mkern4mu \d y \,  c^n_s(y) \check{y}^{\lambda}\hat{y}^{-\alpha} \\
&\quad +\lambda(\lambda -1) 2^{3- \lambda} \iiint \d z  \d x \d y \, z^2 x^{\lambda -2}  \1_{[2z ,\infty)}(x) \1_{[1+z,\infty)}(x)  \kappa_n[c^n_s](x,y,z)\\
&\le  6^{\lambda}\cdot  \|\varphi\|_{0,2\lambda}\Mom_0(c^n_s)  \Mom_{-\alpha,\lambda}(c^n_s)\\
 &\quad+  \lambda(\lambda -1) 2^{3- \lambda} \biggr( \int_0^n  \d z \,  z^2 \varphi(z) \int_1^n \d x \int_1^n \d y \, x^{\lambda -2 } (x^\mu y^{\nu} + x^{\nu} y^{\mu}) c^n_s(x) c^n_s(y) \\
 &\mkern160mu + \int_0^n \d z \, z^2  \varphi(z) \int_1^n \d x  \int_0^1 \d y \, \1_{[z,\infty)}(y) y^{-\alpha} x^{\lambda-2} (x^{\mu}+x^{\nu}) c^n_s(x) c^n_s(y)   \biggl).
\end{split}\]
In the above estimates, we used Assumption \ref{ass:global} (ii)  to restrict the estimate to the cases for $x > z$ and $y>z$.
In the last step, we split the integral according to whether $y\le 1$ or $y \ge 1$. For $y \ge 1$, we follow the method of~\cite[Lemma 3.2]{si2024existence}, adapted to the continuum setting; for the reader's convenience we record the elementary inequalities used. Assuming without loss of generality that $\mu\ge\nu$, so that $\lambda=\mu$, the conditions $\mu,\nu\in[0,2]$ and $\mu+\nu\le3$ give, for $x,y\ge1$,
\begin{equation}\label{eq:el-317}
x^{\lambda-2}(x^\mu y^\nu+x^\nu y^\mu) = x^{2\lambda-2}y^\nu + x^{\mu+\nu-2}y^\lambda \le x^{2\lambda-2}y^\nu + x y^\lambda \,.
\end{equation}
If $\nu\le1$, the right-hand side of~\eqref{eq:el-317} is in turn bounded by
\begin{equation}\label{eq:el-318}
x^{\lambda-2}(x^\mu y^\nu+x^\nu y^\mu) \le x^{\lambda} y + x y^\lambda,
\end{equation}
and the conservation of the first moment bounds the double integral below by $2\Mom_1(c_0)\Mom_\lambda(c^n_s)$. If instead $\nu>1$, inequality~\eqref{eq:el-317} yields
\begin{equation}\label{eq:el-320}
\int_1^n \d x\int_1^n \d y\, x^{\lambda-2}(x^\mu y^\nu+x^\nu y^\mu)\, c^n_s(x)c^n_s(y) \le \int_1^n \d x\int_1^n \d y\, \bigl(x^{2\lambda-2}y^\nu + x y^\lambda\bigr) c^n_s(x)c^n_s(y),
\end{equation}
and bounding the first term on the right by H\"older's inequality and interpolation between the first and the $\lambda$-th moment, as in~\cite[Lemma 3.2]{si2024existence}, we obtain in either case
\[ \int_1^n \d x \int_1^n \d y \, x^{\lambda -2 } (x^\mu y^{\nu} + x^{\nu} y^{\mu}) c^n_s(x) c^n_s(y) \le 2 C_L (1+2 \Mom_\lambda (c^n_s) ),\] with $C_L = \max\{\Mom_1(c_0)^{\frac{2-\min(\nu,\mu)}{\lambda -1}}, \Mom_1(c_0)\}$,
while since $2\lambda -2 \le \lambda$, for $y \le 1$, we have
\[\begin{split}
	\MoveEqLeft\int_0^n \d z \, z^2 \varphi(z) \int_1^n \d x  \int_0^1 \d y \, \1_{[z,\infty)}(y) y^{-\alpha} x^{\lambda-2} (x^{\mu}+x^{\nu}) c^n_s(x) c^n_s(y) \\
	& \le \int_0^n \d z z^{2-\alpha} \varphi(z) \int_1^n \d x \, x^{\lambda-2} (x^{\mu}+x^{\nu}) c^n_s(x)  \int_0^1 \d y \, c^n_s(y) \\
	&\le 2  \|\varphi\|_{0,2-\alpha} \Mom_{0}(c_s^n) \Mom_{\lambda}(c_s^n).
\end{split}\]
In the second integral, we can estimate by monotonicity
\[
	(\Delta_z p_{-\alpha})(x) \le p_{-\alpha}(x-z).
\]
so that, using \eqref{eq:alpha bound K}
\[\begin{split}
	\MoveEqLeft\iiint  \d z \d x \d y\, \1_{[z,1-z]}(x) (\Delta_z p_{-\alpha})(x)  \kappa_n[c^n_s](x,y, z)\\
&\le \iiint  \d z \d x \d y\,\1_{[z,1-z]}(x)  (x-z)^{-\alpha} \kappa_n[c^n_s](x,y, z) \\
&\le  C_\alpha \int_{\bR_+} \d z \,\varphi(z)  z^{-2\alpha}  \int_{\bR_+} \d x \,c^n_s(x)  \int_{\bR_+} \d y \, c^n_s(y)  \check{y}^{\lambda} \hat{y}^{-\alpha}\\
&\le C_\alpha \|\varphi\|_{0,-2\alpha}\Mom_0(c^n_s) \Mom_{-\alpha,\lambda}(c^n_s).
 \end{split}\]
For the third integral, we also use  \eqref{eq:alpha bound K} and get
\begin{equation*}
\begin{split}
\MoveEqLeft\iiint \d z \d x \d y\, \1_{[z \vee (1-z),1+z)}(x) \bra*{(1+2z)^{\lambda} +(x-z)^{-\alpha}} \kappa_n[c^n_s](x,y, z)  \\
&\le \int_{\bR_+} \d z \, \varphi(z)\int_{\bR_+} \d x\, \1_{[z,1+z]}(x) ((1+2z)^{\lambda}\hat{x}^{-\alpha}\check{x}^{\lambda}    + C_\alpha\hat{x}^{-2\alpha}\check{x}^{\lambda} )   c^n_s(x)  \int_{\bR_+} \d y \, c^n_s(y)  \check{y}^{\lambda} \hat{y}^{-\alpha} \\
&\le  C_{\alpha,\lambda}(\| \varphi\|_{0,-2\alpha}+\| \varphi\|_{-\alpha,\lambda}+\| \varphi\|_{0,2\lambda}) \Mom_0(c^n_s)   \Mom_{-\alpha,\lambda}(c^n_s)
\end{split}
\end{equation*}
Combining the cases, we have
\begin{align*}
\MoveEqLeft \iiint  \d z \d x \d y \, (\Delta_z h)(x)  \kappa_n[c^n_s](x,y, z)  \le  4\cdot 3^\lambda \|\varphi\|_{0,2\lambda}\Mom_0(c^n_s)  \Mom_{-\alpha,\lambda}(c^n_s)\\
 &+  \lambda(\lambda -1) 2^{3- \lambda} \biggr(2 \|\varphi\|_{0,2}  C_L (1+2 \Mom_\lambda (c^n_s) ) +2  \|\varphi\|_{0,2-\alpha} \Mom_{0}(c_s^n) \Mom_{\lambda}(c_s^n)  \biggl)\\
 &+ C_\alpha \|\varphi\|_{0,-2\alpha}\Mom_0(c^n_s) \Mom_{-\alpha,\lambda}(c^n_s) \\
& + C_{\alpha,\lambda}\bra*{ \|\varphi\|_{-2\alpha,2\lambda}} \Mom_0(c^n_s)   \Mom_{-\alpha,\lambda}(c^n_s).
\end{align*}
Finally, by using monotonicity of moment, that is $\Mom_0(c^n_s)=\Mom_0(c^n_0)\le \Mom_0(c_0)$ and $\Mom_{\lambda}(c^n_s) \le  \Mom_{-\alpha,\lambda}(c^n_s)$ as well as by Assumption~\ref{ass:existence} that $\varphi\in L^1_{-2\alpha,2\lambda}$, we conclude
\begin{align*}
\iiint& \d z \d x \d y\, (\Delta_z h)(x)  \kappa[c^n_s](x,y, z)  \le C_{\mu,\nu, \alpha,\varphi, c_0} \bra*{1+ \Mom_{-\alpha,\lambda}(c^n_s)}
\end{align*} for some constant $ C_{\mu,\nu, \alpha, \varphi, c_0}>0$.
Hence with Grönwall's inequality, we obtain the first statement
\[ \Mom_{-\alpha,\lambda}(c^n_s) \le  C_{\mu,\nu, \alpha, \varphi, c_0}t  \exp(   C_{\mu,\nu, \alpha,  \varphi, c_0} t ) \Mom_{-\alpha,\lambda}(c_0)  . \]
For the second statement, let $f \in L^{\infty}(\bR_+)$ and define $g(x) = f(x) x^{-\alpha}$, note that $(\Delta_z g)(x) \le 4 \|f\|_{L^\infty} (x-z)^{-\alpha}$ so that we use \eqref{eq:alpha bound K} from Assumption~\ref{ass:existence} to obtain
\begin{align*}
\MoveEqLeft \int_{1/n}^n  f(x) x^{-\alpha} (c^n_t( x) - c^n_s( x)) \d x = \int_s^t \d r \iiint \d z  \d x \d y\, (\Delta_z g)(x)  \kappa_n [c^n_r](x,y,z)\\
& \le  4 \|f\|_{L^\infty} \int_0^t \d s \iiint \d z  \d x \d y \,(x-z)^{-\alpha}\kappa_n [c^n_r](x,y,z)\\
&\le 4 C_\alpha \|f\|_{L^\infty}  \int_s^t \d r \int_0^n \d z \varphi(z) \hat{z}^{-\alpha}\int_0^n \d x \, \hat{x}^{-\alpha}\check{x}^{\lambda} c^n_r(x) \int_0^{n} \d y \hat{y}^{-\alpha}\check{y}^{\lambda}  c^n_r(y)\\
&\le 4 C_\alpha\|f\|_{L^\infty} \|\varphi\|_{-\alpha,0}  \int_s^t \d r (\Mom_{-\alpha,\lambda}(c^n_r) )^2 \\
&\le 4C_\alpha \|f\|_{L^\infty} \|\varphi\|_{-\alpha,0}  ( CT e^{CT})^2 (t-s).
\end{align*}
A similar argument using $(\Delta_z f)(x)\le 4 \|f\|_{L^\infty}$ shows that
\begin{equation}
	\int_{1/n}^n  f(x) (c^n_t( x) - c^n_s( x)) \d x   \le 4 \|f\|_{L^\infty} \|\varphi\|_{0}  \int_s^t \d r ( \Mom_{-\alpha,\lambda}(c^n_r))^2.
\end{equation}
Then the second statement follows by noting that the bound is uniform for functions with uniformly bounded $L^\infty$ norm and $\operatorname{sgn} (c^n_t - c^n_s)$ has $L^\infty$ norm $1$.
\end{proof}
The strategy to prove the existence of weak solution according to Definition~\ref{def:weaksolution} is to show weak $Y_{-\alpha,0}$ compactness of the truncated solution.
We combine techniques established for the (generalized continuous) exchange-driven growth model~\cite{barik2025continuous} with others from the Smoluchowski coagulation equation~\cite{laurenccot2014,barik2020mass}.
The established compactness will be upgraded to the space $Y_{-\alpha,\lambda}$. This means we need to show $(c^n(t))_{n > 1}$ is weakly compact in $L^1(\bR_+, \hat{x}^{-\alpha} \d x)$ for each $t\ge 0$ and
$(c^n)_{n>1}$ is weakly equicontinuous as a map in $C([0,T); L^{1}(\bR_+, \hat{x}^{-\alpha} \d x))$. By the Dunford-Pettis theorem \cite[Theorem 2.3, Proposition 2.6]{laurenccot2014}, a subset $\cF$ of $L^1(\bR_+, \hat{x}^{-\alpha}\d x)$ is weakly $L^1$ compact if and only if $\cF$ is uniformly integrable and uniformly tight.
We obtain the uniform integrability via the de la Vallée-Poussin theorem~\cite{DellacherieMeyer1982} and the uniform tightness via the boundedness of a higher moment.
\begin{definition}[De la Vallée-Poussin functions]
Define $\cC_{VP}\subset C^2(\bR_+)$ to be the set of  non-negative, convex functions such that for $\sigma\in \cC_{VP} $, it holds $\sigma(0)=\sigma'(0)=0$, $\sigma'$ is a concave function, $\sigma'(x)>0$ if $x>0$ and is superlinear, that is
\[
	\lim_{x\to \infty}\sigma'(x)=\lim_{x\to \infty}\frac{\sigma(x)}{x}=\infty \,.
\]
\end{definition}
\begin{remark}\label{rmk:CVPfunction}
As a consequence of the de la Vallée-Poussin theorem~\cite{DellacherieMeyer1982, laurenccot2014, barik2025continuous}, we have that for any initial datum $c_0\in Y_{0,\lambda}^+$, there exists $\sigma_{\lambda}\in \cC_{VP}$ such that
\begin{equation}\label{eq:CVPbound:initial}
	\int_{0}^\infty x^{\lambda-1} \sigma_{\lambda}(x) c_0(x) \d x<+\infty \,.
\end{equation}
Hence, by testing with the function $\sigma_\lambda$,
we obtain the propagation of the bound~\eqref{eq:CVPbound:initial} for later times.
\end{remark}
We collect some properties of de la Vallée-Poussin functions, the  proofs of which can be found in \cite[Proposition 2.14]{laurenccot2014}.
\begin{lemma}\label{lem:bound:sigma}
	Any $\sigma \in \mathcal{C}_{VP}$ satisfies for $x,r\geq 0$ the following inequalities
	\begin{subequations}
	\begin{align}
		0 &\le \sigma(x)\le x \sigma'(x) \le 2 \sigma(x), \label{eq:bound:sigma:1} \\
		0 &\le \sigma(rx)\le \max\{1,r^2\}\sigma(x), \\
		0 &\le x \sigma''(x)\le \sigma'(x) ,
	\end{align}
	\end{subequations}
	and
	\begin{align}\label{eq:bound:sigma:vartheta}
		x (\sigma'(y)-\sigma'(x)) \le \vartheta(y)-\vartheta(x),\quad \text{ with } \vartheta(x)= x \sigma'(x)- \sigma(x) \quad\text{ for $x,y \ge 0$.}
	\end{align}
\end{lemma}
\begin{remark}
In fact, $\eqref{eq:bound:sigma:vartheta}$  is a direct consequence of the convexity of $\sigma$. Moreover, we will use that
\begin{equation}\label{eq:compare-theta-sigma}
0 \le \vartheta(x) \le \sigma(x) \text{ for } x \ge 0,\end{equation} which follows directly from $\eqref{eq:bound:sigma:1}$.
\end{remark}
Another technical tool is the product rule for the discrete Laplacian
\begin{equation}\label{eq:id:discretelaplace}
	\bigl(\Delta_z (fg)\bigr)(x)= (\Delta_z f)(x)\, g(x) + f(x+z) \partial_z^+ g(x) - f(x-z) \partial_z^-  g(x)
\end{equation}
where $\partial_z^+ g(x) = g(x+z) -g(x)$ and $\partial_z^- g(x) = g(x) -g(x-z)$.

Our arguments for uniform integrability are an extension of~\cite[Lemma 3.5]{barik2025continuous} to cope with the singularity of the kernel at zero.
The argument for uniform integrability needs to cope with the possible growth of the kernel at infinity as well as its singularity at zero.
Although the proof is quite technical, the main idea is to use a change of variable and integration by parts to apply the discrete Laplacian to $K_n$ to make use of the bound on the second derivatives~\eqref{ass:2.derivative,K} from Assumption~\ref{ass:existence}.
\begin{proposition}[Uniform integrability]\label{prop:UIbdd}  Assume $K$ satisfies Assumptions  \ref{ass:existence}.
For $c_0 \in Y_{\lambda}^+$ let $\sigma\in \cC_{VP}$ be such that
\[
\int_{0}^\infty   \sigma(\hat{x}^{-\alpha} c_0(x) )\d x<+\infty .
\]
Let $(c^n)_{n>1}$ solve the weak truncated equation~\eqref{eq:weak-truncated} starting from the truncated $c_0$.
Then for each $T\in (0,\infty)$, there exists $C(T)=C_{\mu,\nu,\alpha,\varphi,c_0,T}>0$ such that
\[\sup_{t \in [0,T]}\sup_{n}  \int_{1/n}^{n} \sigma(\hat{x}^{-\alpha} c^n_t(x)) \d x  \le C(T). \]
\end{proposition}

\begin{proof}
	From the weak formulation \eqref{eq:weak-truncated}, we have
\begin{align*}\frac{\d}{\d t} \int_{1/n}^n \sigma(\hat{x}^{-\alpha}c^n_t(x) )\d x
	&= \int^n_{1/n} \sigma'(\hat{x}^{-\alpha}  c^n_t(x)) \hat{x}^{-\alpha} \frac{\d}{\d t} c^n_t(x) \d x\\
& = \iiint \d z \d x \d y\, \bigl(\Delta_z(\sigma'(u^n_t) \hat{p}_{-\alpha})\bigr)(x) \kappa_n [c^n_t](x,y,z)
\end{align*} where $\hat{p}_{-\alpha}(x) = \hat{x}^{-\alpha}$ and $u^n_t(x) = \hat{p}_{-\alpha}(x) c^n_t(x) $.
At this point, we use the crucial inequality~\eqref{eq:bound:sigma:vartheta} to obtain the elementary bound
\[
u^n_t(x) \bigl(\Delta_z(\sigma'(u^n_t))\bigr)(x)\le\bigl(\Delta_z(\vartheta(u^n_t) )\bigr)(x) \,.
\]
Using that $\hat{p}_{-\alpha}$ is non-increasing and $\sigma' \ge 0$, we can drop the second term in the next line, use the above estimate, the definition of $\vartheta$,
estimate~\eqref{eq:bound:sigma:1} and ~\eqref{eq:bound:sigma:vartheta} from Lemma~\ref{lem:bound:sigma} to bound
\[\begin{split}
&\begin{multlined}[\textwidth]
	c^n_t(x) \bigl(\Delta_z(\sigma'(u^n_t) \hat{p}_{-\alpha})\bigr)(x)=c^n_t(x)\Bigl( \bigl(\Delta_z(\sigma'(u^n_t))\bigr)(x)  \hat{p}_{-\alpha}(x) \\
	+ \sigma'(u^n_t)(x+z) ( \hat{p}_{-\alpha}(x+z)-\hat{p}_{-\alpha}(x))
	 + \sigma'(u^n_t)(x-z) (\hat{p}_{-\alpha}(x-z)-\hat{p}_{-\alpha}(x))\Bigr)
\end{multlined}\\
&\begin{multlined}[\textwidth]
	\le c^n_t(x)\Bigl( (u^n_t(x))^{-1} \hat{p}_{-\alpha}(x) \bigl(\Delta_z(\vartheta(u^n_t) )\bigr)(x) + \bigl(\sigma'(u^n_t)(x-z) \\
	-  \sigma'(u^n_t)(x)\bigr)\bigl(\hat{p}_{-\alpha}(x-z)-\hat{p}_{-\alpha}(x)\bigr)
+\sigma'(u^n_t)(x)\bigl(\hat{p}_{-\alpha}(x-z)-\hat{p}_{-\alpha}(x)\bigr)\Bigr)
\end{multlined}\\
&\begin{multlined}
	\le \bigl(\Delta_z(\vartheta(u^n_t) )\bigr)(x) +c^n_t(x) (u^n_t(x))^{-1} \bigl(\vartheta(u^n_t(x-z))-\vartheta(u^n_t(x))\bigr) \bigl(\hat{p}_{-\alpha}(x-z)-\hat{p}_{-\alpha}(x)\bigr) \\
+ 2c^n_t(x)(u^n_t(x))^{-1}  \sigma(u^n_t)(x)\bigl(\hat{p}_{-\alpha}(x-z)-\hat{p}_{-\alpha}(x)\bigr)
\end{multlined}\\
&\begin{multlined}
	\le  \bigl(\Delta_z(\vartheta(u^n_t) )\bigr)(x) +  \bigl(\vartheta(u^n_t(x-z))-\vartheta(u^n_t(x))\bigr) \bra*{\frac{\hat{p}_{-\alpha}(x-z)}{\hat{p}_{-\alpha}(x)}-1} \\
+ 2   \sigma(u^n_t)(x)\bra*{\frac{\hat{p}_{-\alpha}(x-z)}{\hat{p}_{-\alpha}(x)}}\1_{ [0,1] }(x-z).
\end{multlined}
\end{split}
\]
 By defining $g_{\alpha}(x,z):=\bra*{\frac{\hat{p}_{-\alpha}(x-z)}{\hat{p}_{-\alpha}(x)}-1}$, we have the splitting into
 \begin{align}\frac{\d}{\d t} \int_{1/n}^n \sigma(\hat{x}^{-\alpha} c^n_t(x) )\d x
& \le  \iiint  \d z \d x \d y \biggl(  \bigl(\Delta_z(\vartheta(u^n_t))\bigr)(x)  \tag{I}\label{eq:secondorder,sigma}
\\
&\quad+  \bigl(\vartheta(u^n_t(x-z))-\vartheta(u^n_t(x))\bigr) g_{\alpha}(x,z)\label{eq:firstorder,sigma}\tag{II}\\
&\quad +  2   \sigma(u^n_t)(x) \frac{\hat{p}_{-\alpha}(x-z)}{\hat{p}_{-\alpha}(x)} \1_{ [0,1] }(x-z) \biggr)K_n(x,y,z) c^n_t(y) .\tag{III}\label{eq:zerothorder,sigma}\end{align}
The integral~\eqref{eq:secondorder,sigma} is bounded using the assumption~\eqref{ass:2.derivative,K} by
\begin{equation}\begin{split}
\MoveEqLeft \iiint \d z \d x \d y  \, c^n_t(y)   \vartheta(u^n_t(x)) \bigl(\Delta_z( K_n(\cdot,y,z))\bigr)(x)\\
&\le \iiint \d z \d x \d y  \, c^n_t(y)   \vartheta(u^n_t(x)) z^2 \|\partial_1^2 K_n(\cdot ,y,z)\|_{L^\infty} \\
&\le  \int_{\bR_+} \d z \,  z^2 \varphi(z) \int_{\bR_+} \d y \, \hat{y}^{-\alpha}\check{y}^{\lambda} c^n_t(y)    \int_{1/n}^n \d x\,   \vartheta(u^n_t(x)) \,.
\end{split}\end{equation}
In the integral \eqref{eq:firstorder,sigma}, we can drop since $\vartheta(x)\ge 0$ a negative term and change the variable
\[ \begin{split}
	\MoveEqLeft \iiint  \d z \d x \d y\, \bigl(\vartheta(u^n_t(x-z))-\vartheta(u^n_t(x))\bigr) g_{\alpha}(x,z) K_n(x,y,z) c^n_t(y) \\
& \le \iiint \d z \d x \d y \, \vartheta(u^n_t(x)) g_{\alpha}(x+z,z) K_n(x+z,y,z) c^n_t(y).\end{split}\]
Since for $x \ge 1$, it holds $g_{\alpha}(x+z,z)=0$, we only have to consider the case $x\le 1$, $x+z \le n$ and estimate
 \[\begin{split}
 g_{\alpha}(x+z,z)  K_n(x+z,y,z)
 &= \bigl(\hat{p}_{-\alpha}(x) - \hat{p}_{-\alpha}(x+z)\bigr)\hat{p}_{\alpha}(x+z) K_n(x+z,y,z)\\
 &\le \hat{p}_{-\alpha}(x) \hat{p}_{\alpha}(x+z) K_n(x+z,y,z)\\
 &\le \widecheck{(x+z)}^{\lambda} \hat{p}_{-\alpha}(z)   \hat{y}^{-\alpha} \check{y}^{\lambda} \varphi(z)\\
 &\le 2^\lambda  z^{-\alpha}   \hat{y}^{-\alpha} \check{y}^{\lambda} \varphi(z) \,, \end{split}\]
 where in the last inequality, we used $\widecheck{2x} \le 2 \check{x}$ and  Assumption~\ref{ass:global}~(ii).
Together, we can estimate the integral~\eqref{eq:firstorder,sigma} by
 \begin{equation*}\begin{split}
\MoveEqLeft \iiint \d z \d x \d y \, \vartheta(u^n_t(x)) g_{\alpha}(x+z,z) K_n(x+z,y,z) c^n_t(y)\\
&\le  2\iiint \d z \d x \d y \, \1_{[0,1]}(x) \vartheta(u^n_t(x))  z^{-\alpha} \varphi(z) \hat{y}^{-\alpha}\check{y}^\lambda   c^n_t(y)\\
&\le  2 \int_{\bR_+} \d z \, z^{-\alpha}\varphi(z)\int_{\bR_+} \d y \, \hat{y}^{-\alpha} \check{y}^\lambda  c^n_t(y) \int_{1/n}^n \d x \, \vartheta(u^n_t(x)).
\end{split}\end{equation*}
The last integral~\eqref{eq:zerothorder,sigma} is estimated using \eqref{eq:alpha bound K} from Assumption~\ref{ass:existence} by
 \begin{equation*}\begin{split}
 \MoveEqLeft \iiint  \d z \d x \d y\, 2   \sigma(u^n_t)(x)\hat{p}_{-\alpha}(x-z)(\hat{p}_{-\alpha}(x))^{-1} \1_{ [0,1] }(x-z) K_n(x,y,z) c^n_t(y) \\
 &\le 2 C_{\alpha} \iiint \d z \d x \d y\,  \varphi(z)   \sigma(u^n_t)(x) \hat{x}^{-\alpha}\check{x}^{\lambda} \1_{ [0,1] }(x-z) \1_{[1/n,n]}(x) \hat{y}^{-\alpha}\check{y}^{\lambda}c^n_t(y)\\
&\le 2 C_{\alpha}\iiint \d z \d x \d y\, (1+z)^{\lambda} \varphi(z)  z^{-\alpha} \sigma(u^n_t)(x) \1_{[1/n,n]}(x) \hat{y}^{-\alpha}(\check{y}^{\lambda}) c^n_t(y) \\
&\le 2^{1+\lambda} C_{\alpha} \int_{\bR_+} \d z \, (z^{-\alpha}+z^{\lambda-\alpha}) \varphi(z) \int_{\bR_+}\d y \, \hat{y}^{-\alpha}\check{y}^{\lambda} c^n_t(y) \int_{1/n}^n \d x \, \sigma(u^n_t)(x) .
\end{split}\end{equation*}
Recalling \eqref{eq:compare-theta-sigma}, we combine the considerations above to finally conclude
\begin{align*}
\frac{\d}{\d t} \int_{1/n}^n \sigma(\hat{x}^{-\alpha} c^n_t(x) )\d x
& \le  C_{\alpha,\lambda} \Mom_{-\alpha,\lambda}(c_t^n) ( \|\varphi\|_{-\alpha,\lambda-\alpha}+ \|\varphi\|_{0,2}+ \|\varphi\|_{0,-\alpha}) \int_{1/n}^n  \sigma(u^n_t(x)) \d x.
\end{align*}
Hence by  Lemma~\ref{lem:prop-lambda-moment} and Grönwall inequality, we have the claim.
\end{proof}

\begin{remark}
 We note that for any initial datum $c_0\in Y_{-\alpha}^+$, there exists $\sigma\in \cC_{VP}$ such that
\begin{equation}
	\int_{0}^\infty  \sigma(x^{-\alpha} c_0(x)) \d x<+\infty \,.
\end{equation} Therefore, the assumption of Proposition \ref{prop:UIbdd} holds for $c_0\in Y_{-\alpha,\lambda}^+$.
\end{remark}

Now we turn to the boundedness of a higher moment, which guarantees tightness for the solutions.
\begin{proposition}[Boundedness of higher moments]\label{prop:higher-moment-bdd} Let $T\in (0,\infty)$. Assume $K$ satisfies Assumptions \ref{ass:existence}.
Let $c_0\in Y_{-\alpha,\lambda}^+$ and $\sigma_{\lambda}\in \cC_{VP}$ be such that~\eqref{eq:CVPbound:initial} from Remark \ref{rmk:CVPfunction} holds, then there exists $C(T)=C_{\mu,\nu,\alpha,\varphi,c_0,T}>0$ such that all $t\in[0,T]$ it holds
\[\int_{1/n}^n x^{\lambda-1} \sigma_{\lambda}(x) c_{t}^n(x) \d x\le C(T).  \]
\end{proposition}
\begin{proof}
Let $h(x)=x^{\lambda-1}\sigma_{\lambda}(x)$. Then, by the weak truncated form~\eqref{eq:weak-truncated}, we have
\begin{align*}\frac{\d}{\d t}& \int_{1/n}^n h(x) c^n_t (x)\d x = \int^n_{1/n}  h(x) \frac{\d}{\d t} c^n_t(x) \d x = \iiint\d z \d x \d y\, (\Delta_z h)(x) \kappa_n [c^n_t](x,y,z).
\end{align*}
By the properties of the function $\sigma_\lambda\in \cC_{VP}$ from \cite[Lemma 3.4]{si2024existence}, we get that $h'$ is increasing so that
\[\begin{split}
(\Delta_z h)(x)& = \int^{x+z}_x h'(y) \d y - \int^x_{x-z} h'(y) \d y
\le z \bigl( h'(x+z)-h'(x-z)\bigr)=z \int_{x-z}^{x+z} h''(y) \d y\\
&=z \int_{x-z}^{x+z}  \Bigl[ (\lambda-1)(\lambda-2) y^{\lambda-3} \sigma_{\lambda}(y) + 2 (\lambda-1) y^{\lambda -2} \sigma_{\lambda}' (y) + y^{\lambda-1}\sigma_{\lambda}''(y) \Bigr] \d y\\
&\le z  (2\lambda-1) \int_{x-z}^{x+z} y^{\lambda-2}  \sigma_{\lambda}' (y) \d y,
\end{split}\]
where we used  $(\lambda-1)(\lambda-2) \le 0$ for $\lambda \in (1,2]$.
We use the convexity of $\sigma_\lambda$ and the sublinearity of $z \mapsto z^{\lambda -1}$ to estimate
\begin{equation*}\begin{split}\int_{x-z}^{x+z}   y^{\lambda-2}  \sigma_{\lambda}' (y) \d y &\le  \sigma_{\lambda}' (x+z)  \int_{x-z}^{x+z}y^{\lambda-2} \d y  \le \frac{(x-z+2z)^{\lambda -1} - (x-z)^{\lambda -1 }}{\lambda -1 }\sigma_{\lambda}'(x+z) \\
&\le \frac{2^{\lambda -1}}{\lambda -1} z^{\lambda -1} \sigma_{\lambda}'(x+z) .
\end{split}\end{equation*}
Therefore, it holds
\[(\Delta_z h)(x) \le C_{\lambda}  \sigma_{\lambda}'(x+z)  z^{\lambda} \quad\text{ with }  C_{\lambda}= \frac{ 2^{\lambda-1} \lambda^2}{\lambda-1} . \]
Since $\sigma_{\lambda}\in \cC_{VP}$, we have  \[\frac12 \sigma_{\lambda}'(x+z)\le \frac{ \sigma_{\lambda}(x+z)}{x+z}=  \sigma_{\lambda}\bra*{\frac{x+z}{x} x}\frac{1}{x+z}\le\bra*{1+\bra*{\frac{x+z}{x}}^2}\frac{ 1}{x+z} \sigma_{\lambda}(x)\le 5\frac{1}{x+z}\sigma_{\lambda}(x)\] for $0 \le z \le x$.
For $x\le 1, z\le x$, we estimate as follows
\[\sigma_{\lambda}'(x+z)\le \sigma_\lambda'(1+z)\le \frac{2}{1+z} \sigma_\lambda(1+z)\le 2 (1+z) \sigma_\lambda(1).\] With the case separation on $x\ge 1$ and $x\le 1$, we get
\[ \begin{split}
\MoveEqLeft\begin{multlined}
		 \iiint \d z \d x \d y\, (\Delta_z h)(x) \kappa_n [c^n_t](x,y,z)\le  C_{\lambda}  \iiint\d z \d x \d y\, \1_{[1,\infty)}(x) \sigma_{\lambda}'(x+z)  z^{\lambda} \kappa_n[c^n_t](x,y,z) \\
+ 2 \sigma_\lambda(1)  \iiint \d z \d x \d y\, \1_{[0,1]}(x) (1+z)  z^\lambda \kappa_n[c^n_t](x,y,z)
\end{multlined}\\
&\begin{multlined} \leq 10 C_{\lambda}  \iiint\d z \d x \d y\,  \1_{[1,\infty)}(x) \frac{1}{x+z} \sigma_{\lambda}(x)  z^{\lambda} \kappa_n [c^n_t](x,y,z)\\
+ 2 \bigl( \|\varphi\|_{0,\lambda-\alpha+1}+ \|\varphi\|_{0,\lambda-\alpha}\bigr)\sigma_\lambda(1) \int_{0}^{1}\d x  c^n_t(x) \int_{\bR_+}\d y\, \hat{y}^{-\alpha}  \check{y}^{\lambda}  c^n_t(y)
\end{multlined}\\
&\begin{multlined} \le 10 C_{\lambda} \|\varphi\|_{0,\lambda} \Mom_{-\alpha,\lambda}(c^n_t) \int_{\bR_+} x^{\lambda-1} \sigma_{\lambda}(x)  c^n_t(x)  \d x \\
+ 4 \|\varphi\|_{0,\lambda-\alpha+1}\sigma_\lambda(1) \Mom_0(c^n_0) \Mom_{-\alpha,\lambda}(c^n_t).
\end{multlined}
\end{split}
\]
By Lemma \ref{lem:prop-lambda-moment}, we have the uniform bound for  $\Mom_{-\alpha,\lambda}(c^n_t)$ for  $t\le T$. Hence, we conclude
\begin{align*}\frac{\d}{\d t}& \int_{1/n}^n h(x) c^n_t (x)\d x \le C_{\mu,\nu,\alpha,\varphi,c_0}(T) \biggl(1+\int_{1/n}^n h(x) c^n_t(x) \d x\biggr)\end{align*}
and the claim follows from Grönwall's inequality.
\end{proof}
\begin{proposition}[$Y_\lambda$-weak subsequence convergence to strongly $L^1$ continuous limit]\label{prop:convergence}Assume $K$ satisfies Assumptions  \ref{ass:existence}.  Let $T\in (0,\infty)$.   We have
$c^n \to c$ in $C([0,T], w\mhyphen Y_{-\alpha,\lambda} )$ along a subsequence and $c \in C([0,T];L^1(\bR_+,\hat{x}^{-\alpha}\d x))$.
\end{proposition}
\begin{proof}
The estimate in Proposition \ref{prop:higher-moment-bdd} implies for $k\ge 1$ the bound
\[\begin{split}\int_{k}^{\infty} c^n_t(x) \hat{x}^{-\alpha}\d x & \le \int_{k}^{\infty} c^n_t(x) \d x
\le k^{1-\lambda}(\sigma_{\lambda}(k))^{-1} \int_{k}^\infty x^{\lambda -1} \sigma_\lambda(x) c^n_t(x) \d x\\
& \le C(T) k^{1-\lambda}(\sigma_{\lambda}(k))^{-1}\to 0 \quad \text{ as } k \to \infty.\end{split}\]
Hence, the sequence $(\hat{p}_{-\alpha} c^n_t)_{n>1, t\in[0,T]}$ is uniformly tight with respect to the measure  $\hat{x}^{-\alpha }\d x$.
Moreover, Proposition~\ref{prop:UIbdd} implies $( c^n_t)_{n>1}$ is also uniformly integrable with respect to $\hat{x}^{-\alpha} \d x$ for each $t\in [ 0,T]$ by the de la Vallée-Poussin theorem.
We conclude via the Dunford-Pettis theorem that $(c^n_t)$ is relatively weakly sequentially compact in $L^1(\bR_+, \hat{x}^{-\alpha} \d x)$ for each $t\in[0,T]$.
Furthermore, Lemma~\ref{lem:prop-lambda-moment} implies $c^n$ is strongly equicontinuous in $L^1(\bR_+, \hat{x}^{-\alpha} \d x)$ for $t\in [0,T]$. In particular, it is weakly equicontinuous. Thus, by a variant of the Arzel\'{a}-Ascoli theorem, we obtain non-negative $c\in C([0,T], w\mhyphen L^1(\bR_+, \hat{x}^{-\alpha} \d x))$ along a subsequence.

By a standard truncation argument (see e.g.~\cite[Proof of Theorem 2.3]{barik2025continuous}), we get that the weak convergent limit satisfies for all $l > 0$ and $t\in[0,T]$ the bound
\[\int_{0}^l x^{\lambda-1} \sigma_\lambda(x)c_t(x)\d x \le \lim_{n \to \infty} \int_{0}^l x^{\lambda-1} \sigma_\lambda(x)c^n_t(x) \d x\le C(T).\]
Fatou's lemma implies that by letting $l\to \infty$ the bound
\[
	\sup_{t\in[0,T]}\int_0^\infty x^{\lambda-1} \sigma_\lambda(x)c_t(x)\d x \le C(T).
\]
Now, we consider for $g \in L^\infty(\bR_+)$, $t\in[0,T]$ and $l\ge1$ the difference
\[\begin{split}
	\biggl|\int_0^\infty   &g(x) (x^{-\alpha}+x^{\lambda}) [c^n_t(x)-c_t(x)]  \d x \biggr| \\
&\le \biggl|\int_0^l   g(x) (x^{-\alpha}+x^{\lambda}) [c^n_t(x)-c_t(x)]  \d x \biggr| + \biggl|\int_l^\infty   g(x) (x^{-\alpha}+x^{\lambda}) [c^n_t(x)-c_t(x)]  \d x \biggr| \,.
\end{split}\]
We rewrite the first term as
\begin{equation}\begin{split}
\int_0^l  g(x) (x^{-\alpha}+x^{\lambda}) [c^n_t(x)-c_t(x)]  \d x  = \int_0^l   g(x) (1+ x^{\lambda+\alpha}) [c^n_t(x)-c_t(x)] x^{-\alpha}  \d x \,.
\end{split}
\end{equation}
Since $g (1+ p_{\lambda+\alpha}) \1_{[0,l]} \in L^{\infty}(\bR_+, {x}^{-\alpha} \d x )$, we obtain its convergence to zero as $n\to \infty$ due to its weak convergence in $L^1((0,l),x^{-\alpha} \d x)$.
We estimate the second term as follows
\[\begin{split}
	\biggl|\int_l^\infty \mkern-8mu   g(x) (x^{-\alpha}+x^{\lambda}) [c^n_t(x)-c_t(x)]  \d x \biggr|
&\le (1+ l^{-\lambda -\alpha}) \int_l^\infty   |g(x)| \, x^{\lambda}\, |c^n_t(x)-c_t(x)|  \d x\\
&\le 2 \|g\|_{L^\infty}\sup_{y\ge l}\frac{y}{\sigma_{\lambda}(y)} \int_l^\infty  x^{\lambda-1}  \sigma_{\lambda}(x) \bigl(c^n_t(x)+c_t(x)\bigr) \d x \\
& \le  4\|g\|_{L^\infty}\sup_{y\ge l}\frac{y}{\sigma_{\lambda}(y)}C(T) \to 0 \quad \text{ as } l \to \infty \text{ uniformly in } n.
 \end{split}\]
Hence, we conclude $\lim_{n\to\infty}|\int_0^\infty   g(x) (x^{-\alpha}+x^{\lambda}) [c^n_t(x)-c_t(x)]  \d x|=0 $ for each $g\in L^\infty(\bR_+)$ and $c^n_t \to c_t$ in $w\mhyphen Y_{-\alpha, \lambda}$.

To conclude the time continuity, we use the weak convergence in $L^1(\bR_+,  \hat{x}^{-\alpha} \d x)$ of $c^n_t-c^n_s \to c_t-c_s$ and the $L^1$-equicontinuity from \eqref{eq:equicontinuity} proven in Lemma~\ref{lem:prop-lambda-moment}, to get
\begin{equation*}\begin{split}
		\|c(t)-c(s)\|_{-\alpha,0}&=\sup_{g\in L^\infty, \|g\|_{L^\infty}=1} \biggl|\int_0^\infty g(x) (c_t(x) -c_s(x)) \hat{x}^{-\alpha}\d x \biggr| \\
&=\sup_{g\in L^\infty, \|g\|_{L^\infty}=1} \lim_{n\to\infty} \biggl|\int_0^\infty g(x) (c^n_t(x) -c^n_s(x))  \hat{x}^{-\alpha} \d x \biggr| \\
&\le C(T)(t-s)
 \end{split}\end{equation*} which shows $c \in C([0,T], L^1(\bR_+, \hat{x}^{-\alpha}\d x))$.
\end{proof}
Having identified a limit, we still need to show that the limit satisfies the weak form~\eqref{eq:swCGEDG}.
\begin{proposition}[Identification of limit] \label{prop:limit-equation} Assume $K$ satisfies Assumptions   \ref{ass:existence}.
The subsequence limit $c$ in $Y_{-\alpha,\lambda}$ from  Proposition \ref{prop:convergence}  is a weak solution to~\eqref{eq:CGEDG} on $[0,\infty)$.
\end{proposition}
\begin{proof}
We use similar arguments as
in~\cite[Proof of Theorem 2.3]{barik2025continuous} and
in~\cite[Theorem 2.2]{si2024existence} to show the weak limit $c$ satisfies the weak form of~\eqref{eq:CGEDG}.
Let $f\in L^\infty(\bR_+)$ be a test function for the weak form~\eqref{eq:swCGEDG}.
Then, we have for each $t\in [0,T]$, $n,k\in\bN$ and $n > k > 1 $ the identity
\begin{equation*}\begin{split}\int_{1/n}^n & f(x) (c^n_t( x) - c^n_0( x)) \d x\\
&= \int_0^t \d s \int_{0}^\infty \d z \biggl( \iint_{\mathclap{(1/k,k)^2}} \d x \d y \, (\Delta_z f)(x) \kappa_n[c^n_s](x,y,z) +  \iint_{\mathclap{\bR_+^2\setminus (1/k,k)^2}} \d x \d y \, (\Delta_z f)(x) \kappa_n[c^n_s](x,y,z)\biggr).
\end{split}\end{equation*}
For $z \in \bR_+$, $n>k>1$, $x,y \in (1/k,k)$, we have $|K_n(x,y,z)|\le \varphi(z) k^{2(\lambda+\alpha)}$ and $K_n(x,y,z)\to K(x,y,z)$ pointwise as $n\to\infty$.
By applying \cite[Proposition 2.18]{laurenccot2014}, we have $c^n_s\to c_s$ weakly in $L^1((1/k,k))$ for each $s \in [0,T]$.
Hence, we get the convergence
\[
\lim_{n \to \infty} \int_{(1/k,k)} \mkern-12mu \d y \, K_n(x,y,z)c^n_s(y)  = \int_{(1/k,k)}\mkern-12mu \d y \, K(x,y,z)c_s(y) \text{ for each } x\in (1/k,k), z\in \bR_+, s \in [0,T].
\]
By the estimate $| \int_{(1/k,k)} \d y \, K_n(x,y,z)c^n_s(y) | \le \varphi(z) k^{2(\lambda+\alpha)}\Mom_0(c^n_s) = \varphi(z) k^{2(\lambda+\alpha)}\Mom_0(c^n_0) $ and the bound
$|(\Delta_z f) (x)|\le 4 \|f\|_{L^\infty}$ as well as the weak convergence of $c^n_s$, we have thanks to~\cite[Proposition 2.18]{laurenccot2014} for each $z\in \bR_+$ and $s \in [0,t]$ the convergence
\[
\lim_{n\to\infty}   \int_{(1/k,k)^2}\d x \d y \, (\Delta_z f) (x)\kappa_n[c^n_s](x,y,z) = \int_{(1/k,k)^2} \d x\d y \, (\Delta_z f) (x)\kappa[c_s](x,y,z) ,
\]
Because we can bound
\[
	\biggl| \iint_{(1/k,k)^2} \d x \d y \, (\Delta_z f)(x) \kappa_n[c^n_s](x,y,z)\biggr| \le  4 \|f\|_{L^\infty} \varphi(z) k^{2(\lambda+\alpha)}(\Mom_0(c^n_0))^2, \quad\text{ for }s\in [0,t], z\in\bR_+\]
and this upper bound is integrable in $\{(s,z) \in [0,t] \times \bR_+\}$, we can apply the dominated convergence theorem to obtain
\[\begin{split}\lim_{n\to\infty}\int_0^t \d s \int_0^\infty \d z \iint_{(1/k,k)^2} &\d x \d y \, (\Delta_z f)(x) \kappa_n[c^n_s](x,y,z)\\
&=\int_0^t \d s \int_0^\infty \d z \iint_{(1/k,k)^2} \d x \d y \, (\Delta_z f)(x) \kappa[c_s](x,y,z).\end{split} \]
On the other hand, we can show the remaining terms vanish to zero uniformly in $n$ as $k\to \infty$.
Indeed, by using Lemma~\ref{lem:prop-lambda-moment} and Proposition~\ref{prop:higher-moment-bdd}, we have
\begin{align*}
	&\int_0^t \d s \int_{0}^\infty \d z  \iint\nolimits_{\bR_+^2 \setminus (1/k,k)^2} \d x \d y \, (\Delta_z f)(x) \kappa_n[c^n_s](x,y,z) \\
&\le
 \int_0^t \mkern-2mu \d s \int_{0}^\infty \mkern-4mu \d z\iint  \d x \d y \, \bigl|(\Delta_z f)(x)\bigr| \bigl(\1_{[k,\infty)}(x) \! +\! \1_{[k,\infty)}(y) \! +\! \1_{[0,1/k]}(x) \! +\! \1_{[0,1/k]}(y)\bigr)  \kappa_n[c^n_s](x,y,z) \\
& \le  8  \int_0^t \mkern-2mu \d s \int_{0}^\infty \mkern-4mu \d z\iint  \d x \d y   \bra*{\1_{[0,1/k]}(x)+\1_{[k,\infty)}(x)} \kappa_n[c^n_s](x,y,z)\\
& \le  8  \int_0^t \mkern-2mu \d s \int_{0}^\infty \mkern-4mu \d z \, \varphi(z) \int  \d y \, \hat{y}^{-\alpha} \check{y}^\lambda c^n_s(y)\bra*{k^{-\alpha} z^{-\alpha} \int_{0}^{1/k} x^{-\alpha }c^n_s(x) + \int_{k}^\infty x^{\lambda} c^n_s(x) }\\
& \le  8  \|\varphi\|_{-\alpha,0}\sup_{s\in[0,T]} \Mom_{-\alpha,\lambda}(c^n_s) \int_0^t \d s \bra*{ k^{-\alpha}  \Mom_{-\alpha}(c^n_s) + \sup_{y \ge k}\frac{y}{\sigma_{\lambda}(y)} \int_{k}^\infty \d x \, x^{\lambda-1}\sigma_{\lambda}(x) c^n_s(x) }\\
& \le 8  \|\varphi\|_{-\alpha,0}  C(T)\bra*{ k^{-\alpha}+ \sup_{x \ge k}\frac{x}{\sigma_{\lambda}(x)}} .
 \end{align*}
Hence, by taking $n\to \infty$ and then $k\to \infty$, we have shown the convergence of the right-hand side in the weak form~\eqref{eq:swCGEDG}, that is
\[\lim_{n\to \infty}\int_0^t\d s \iiint \d z \d x \d y  (\Delta_z f)(x) \kappa_n[c^n_s](x,y,z) = \int_0^t\d s \iiint \d z \d x \d y  (\Delta_z f)(x) \kappa[c_s](x,y,z). \]
Likewise, the weak $L^1$ convergence for $t\in [0,T]$ implies the convergence of the left-hand side, that is
\[\lim_{n\to \infty}\int_0^\infty  f(x) (c^n_t( x) - c^n_0( x)) \d x= \int_0^\infty f(x) (c_t( x) - c_0( x)) \d x \]
for each $f\in L^\infty$. Hence the limit $c$ satisfies the weak form~\eqref{eq:swCGEDG}.
\end{proof}
We show that the so constructed solutions conserve the zeroth and first moment, which concludes the proof of Theorem~\ref{thm:existence}.
\begin{proposition}\label{prop:conservation-of-mass-1.moment}  Assume $K$ satisfies Assumptions  \ref{ass:existence}. The weak solution to~\eqref{eq:CGEDG} on $[0,\infty)$ in $Y_{-\alpha,\lambda}$ constructed in Proposition \ref{prop:limit-equation} conserves the mass and the first moment.
\end{proposition}
\begin{proof}
The conservation of mass follows from the definition of weak solutions in~\eqref{eq:swCGEDG} by taking the admissible test-function $\1_{\bR_+}\in L^\infty(\bR_+)$.
For the conservation of the first moment, we consider for $k >1$ the truncated test-function $f_k(x) = x \1_{[0,k]}(x)$ and we get
\[\begin{split}\biggl|\int_{\bR_+}x (c_t-c_0)(x) \d x\biggr|
	&\le \biggl| \int_{0}^k x (c_t - c_t^n)(x) \d x\biggr|+ \biggl| \int_0^{k} x (c_t^n-c_0^n)(x) \d x \biggr|\\
 &\quad + \biggl|\int_0^k  x (c_0^n -c_0)(x) \d x \biggr|+ \biggl|\int_k^\infty x (c_t-c_0)(x) \d x \biggr|. \end{split}\]
By the weak convergence of $c_t^n \to c_t$ in $L^1((0,k))$, the first and the third integral converge to zero as $n\to \infty$.  Now using that the first moment is conserved for the truncated system and bounded  $\lambda$ moment in Lemma \ref{lem:prop-lambda-moment}, we have
\[ \biggl|\int_0^{k} x (c_t^n-c_0^n)(x) \d x\biggr| = \biggl|\int_k^n  x (c_t^n-c_0^n)(x) \d x\biggr| \le \frac{1}{k^{\lambda-1}}\int_0^\infty x^\lambda c^n_t(x) + c^n_0(x) \d x \le \frac{C(T)}{k^{\lambda -1}}. \]
Similarly, we obtain
\[\biggl|\int_k^\infty x (c_t-c_0)(x) \d x \biggr|\le \frac{C(T)}{k^{\lambda-1}}. \]
Hence, we can first take the limit $n\to\infty$ and then consider $k\to\infty$ to obtain the convergence $\int_{\bR_+}x c_t(x) \d x = \int_{\bR_+}x c_0(x) \d x $.
\end{proof}

\section{Uniqueness}
The main idea of the proof is to show a Grönwall's estimate for the moment of the difference of two solutions.
\begin{proof}[Proof of Theorem \ref{thm:uniquness}] Let $e_t(x)=c_t(x)-d_t(x)$, where $c_t,d_t$ are two solutions to~\eqref{eq:CGEDG} with the same initial data.
The proof is a Grönwall's argument for the mixed moment of the difference $\Mom_{-\alpha,\lambda}(e(t))= \int_{\bR_+} \hat{x}^{-\alpha}\check{x}^{\lambda}  |e_t(x) | \d x  $. Let $g_t(x) =\hat{x}^{-\alpha} \check{x}^{\lambda} \operatorname{sign}(e_t(x))$. From now on, we drop the time index and write only $e(x)$. Also, the time derivative below should be understood in the weak sense, that is, after integrating both sides in time.
In addition, we need to introduce a truncation parameter $n>2$ at $0$ and $+\infty$.
We start splitting the time derivative
\[\begin{split}
	\MoveEqLeft
	\frac{\d}{\d t} \int_0^\infty  \1_{[1/n,n]}(x) \hat{x}^{-\alpha} \check{x}^{\lambda} |e(x)| \d x
=\int_0^\infty \1_{[1/n,n]}(x) g(x)  \frac{\d}{\d t} {e}(x) \d x\\
&= \iiint \d z  \d x \d y \,\bigl(\Delta_z(g\1_{[1/n,n]})\bigr)(x) \bigl(\kappa[c](x,y,z)-\kappa[d](x,y,z)\bigr) \\
&= \iiint \d z  \d x \d y \,\bigl(\Delta_z(g\1_{[1/n, n]})\bigr)(x) K(x,y,z)(c(x) e(y)+ e(x) d(y) ) .
   \end{split}\]
By~\eqref{eq:id:discretelaplace}, we can further rewrite
\begin{equation}\label{eq:time-der,lambda-diff}\begin{split}
 \MoveEqLeft\frac{\d}{\d t}  \int_0^\infty \1_{[1/n,n]}(x)  \hat{x}^{-\alpha} x^{\lambda} |e(x)| \d x
\\
&=\iiint \d z  \d x \d y\,(\Delta_z g)(x)\1_{[1/n, n]}(x) K(x,y,z)\bigl(c(x) e(y)+ e(x) d(y)\bigr)
\\
&\quad+ \iiint \d z  \d x \d y \, g(x+z) \bra*{\1_{[1/n-z,1/n]}(x)-\1_{[n-z,n]}(x)}  K(x,y,z)(c(x) e(y)+ e(x) d(y) )\\
&\quad+ \iiint \d z  \d x \d y \, g(x-z) \bra*{ \1_{[n,n+z]}(x)-\1_{[1/n,1/n+z]}(x) } K(x,y,z)(c(x) e(y)+ e(x) d(y) ) \\
&\le  \iiint \d z  \d x \d y \, |(\Delta_z g)(x)|\1_{[1/n, n]}(x) K(x,y,z) c(x) |e(y)| \\
&\quad + \iiint \d z  \d x \d y \,  (\Delta_z g)(x)\1_{[1/n, n]}(x) K(x,y,z) e(x) d(y) \\
&\quad + \text{ boundary terms }.
\end{split}\raisetag{70pt}
\end{equation}
First, we show the boundary terms vanish as $n \to \infty$.
For doing so, we define the abbreviation $\Mom_{-\alpha,\lambda}(c,d) = \max(\Mom_{-\alpha,\lambda}(c),\Mom_{-\alpha,\lambda}(d))$ and estimate
 \begin{align}
\MoveEqLeft \int_{\bR_+}\d z\, \varphi(z) \int_{[z,\infty)^2} \mkern-28mu \d x \d y \, |g(x+z)| \1_{[n-z,n]}(x) (\hat{x}^{-\alpha}\hat{y}^{-\alpha}  \check{x}^{\lambda}\check{y}^{\lambda})\bigl(c(x) (c(y)+d(y))+ (c(x)+d(x)) d(y) \bigr) \nonumber\\
&\le  3  \int_{\bR_+}\d z \, \varphi(z)  \int_{n/2}^n \d x \,  \hat{x}^{-\alpha} {(2x)}^{\lambda} (c(x)+d(x))  \bra[\big]{\check{x}^{\lambda} \hat{x}^{-\alpha}  \Mom_{-\alpha,\lambda}(c, d)} \label{eq:domiantedbound} \\
&  \le  3  \int_{\bR_+}\d z\, \varphi(z) \int_{n/2}^n \d x \, 2^{\lambda} {x}^{2\lambda} \bigl(c(x)+d(x)\bigr)    \Mom_{-\alpha,\lambda}(c , d) \nonumber
\end{align}
which tends to zero as $n\to \infty$ if $c_t,d_t \in Y_{2\lambda}$.
Similarly, the next boundary term can be estimated by
 \begin{equation*}\begin{split}
\int_{\bR_+}&\d z\, \varphi(z) \int_{[z,\infty)^2} \mkern-26mu \d x \d y\, |g(x+z)| \1_{[1/n-z,1/n]}(x) (\hat{x}^{-\alpha}\hat{y}^{-\alpha} \check{y}^{\lambda})(c(x) (c(y)+d(y))+ (c(x)+d(x)) d(y) ) \\
&\le  3  \int_{\bR_+}\d z\, \varphi(z)  \int_{1/2n}^{1/n} \d x \,  (x+z)^{-\alpha} x^{-\alpha}   (c(x)+d(x))    \Mom_{-\alpha,\lambda}(c , d)\\
&  \le  3  \int_{\bR_+}\d z\, \varphi(z)\int_{1/2n}^{1/n} \d x \,  {x}^{-2\alpha} (c(x)+d(x))     \Mom_{-\alpha,\lambda}(c , d)
\end{split}\end{equation*}
which tends to zero as $n\to\infty$ if $c_t,d_t \in Y_{-2\alpha}$.
A further boundary term is estimated using \eqref{eq:alpha bound K} from Assumption~\ref{ass:existence} by
\begin{equation*}\begin{split}
\MoveEqLeft \biggl| \iiint\d z \d x \d y  g(x-z) \1_{[n,n+z]}(x) K(x,y,z)\bigl(c(x) e(y)+ e(x) d(y)\bigr) \biggr|\\
& \le \int_{\bR_+} \d z \int_{n}^{n+z} \d x \int \d y \, \widehat{(x-z)}^{-\alpha}\widecheck{(x-z)}^{\lambda} K(x,y,z) \bigl(c(x) |e(y)| + |e(x)| d(y) \bigr) \\
 &\le  \int_{\bR_+} \d z  \, \varphi(z) \int_{n}^{n+z} \d x \int \d y \, C_\alpha \hat{x}^{-2\alpha} \check{x}^{2\lambda}  \hat{y}^{-\alpha} \check{y}^{\lambda} \bigl(c(x) |e(y)| + |e(x)| d(y) \bigr) \\
 &\le 3 C_{\alpha} \int_{\bR_+} \d z \, \varphi(z)\int_{n}^\infty \d x \, x^{2\lambda} \bigl(c(x)+d(x)\bigr)  \Mom_{-\alpha,\lambda}(c , d) \,,
\end{split}\end{equation*}%
which tends to zero as $n\to \infty$ provided that $c,d\in Y^+_{-2\alpha,2\lambda}$.
Finally with \eqref{eq:alpha bound K} from Assumption~\ref{ass:existence}, we have
\begin{equation*}\begin{split}
\MoveEqLeft \biggl|  \iiint\d z \d x \d y g(x-z) \1_{[1/n,1/n+z]}(x) K(x,y,z)\bigl(c(x) e(y)+ e(x) d(y)\bigr) \biggr|\\
 &\le  \int_{\bR_+} \d z \int_{1/n \vee z}^{1/n+z} \d x \int \d y \, ({x-z})^{-\alpha}\check{x}^{\lambda}  K(x,y,z)(c(x) |e(y)| + |e(x)| d(y) )\\
 & \le \int_{\bR_+} \d z \, \varphi(z) \int_{1/n \vee z}^{1/n+z} \d x \int \d y \, C_{\alpha} \hat{x}^{-2\alpha} \check{x}^{2\lambda} \hat{y}^{-\alpha} \check{y}^{\lambda}\bigl(c(x) |e(y)| + |e(x)| d(y) \bigr)\\
 &\le 3 C_{\alpha} \int_{\bR_+} \d z\, \varphi(z)\int_{1/n \vee z}^{1/n+z} \d x \, \hat{x}^{-2\alpha} \check{x}^{2\lambda} \bigl(c(x)+d(x)\bigr)   \Mom_{-\alpha,\lambda}(c , d) ,
\end{split}\end{equation*}
which tends to zero as $n\to \infty$ provided that $c,d\in Y^+_{-2\alpha,2\lambda}$.

Now, we return to the bulk terms in~\eqref{eq:time-der,lambda-diff}.
Since $|(\Delta_z g )(x)| \1_{[0,x]}(z) \le  (2^{\lambda}+3)\check{x}^{\lambda}\widehat{(x-z)}^{-\alpha} $, we apply again~\eqref{eq:alpha bound K} from Assumption~\ref{ass:existence} to estimate
\begin{equation*}\begin{split}
		\MoveEqLeft \iiint\d z \d x \d y \, |(\Delta_zg)(x)| \, K(x,y,z) \, c(x) |e(y)| \\
&\le  (2^{\lambda}+3) C_{\alpha} \int_{\bR_+} \mkern-8mu \d z \, \varphi(z) \int_{\bR_+} \mkern-8mu  \d x\, \hat{x}^{-2\alpha}  \check{x}^{2\lambda}c(x)\int_{\bR_+} \mkern-8mu \d y \, \hat{y}^{-\alpha} \check{y}^{\lambda} |e(y)|\\
&\le (2^{\lambda}+3)C_{\alpha} \|\varphi\|_0   \Mom_{-2\alpha,2\lambda}(c , d)\Mom_{-\alpha, \lambda}(e).
\end{split}\end{equation*}
We note that $(\Delta_z g)(x) e(x) \le  (\Delta_z \hat{p}_{-\alpha} \check{p}_{\lambda})(x) |e(x)|$
and by the discrete chain rule for the discrete Laplacian in~\eqref{eq:id:discretelaplace}, we bound
\[\begin{split}
	(\Delta_z \hat{p}_{-\alpha} \check{p}_{\lambda})(x) &=  (\Delta_z \check{p}_{\lambda})(x) \hat{p}_{-\alpha}(x) + \check{p}_{\lambda}(x+z) \partial_z^+  \hat{p}_{-\alpha}(x) -   \check{p}_{\lambda}(x-z) \partial_z^- \hat{p}_{-\alpha}(x)\\
	&= (\Delta_z \check{p}_{\lambda})(x) \hat{p}_{-\alpha}(x) +  \check{p}_{\lambda}(x+z) (\Delta_z \hat{p}_{-\alpha})(x) \\
	&\qquad - \bigl( \check{p}_{\lambda}(x+z) - \check{p}_{\lambda}(x-z)\bigr) \bigl(\hat{p}_{-\alpha}(x-z) - \hat{p}_{-\alpha}(x) \bigr)\,\\
	&\le (\Delta_z \check{p}_{\lambda})(x) \hat{p}_{-\alpha}(x) +  \check{p}_{\lambda}(x+z) (\Delta_z \hat{p}_{-\alpha})(x).
\end{split}\]
By using \[(\Delta_z \hat{p}_{-\alpha})(x)   = 0  \quad \text{ if } x-z \ge 1,\] we have the estimate
\[\begin{split}
\check{p}_{\lambda}(x+z) (\Delta_z \hat{p}_{-\alpha})(x) &\le \1_{\{x< 1+z\}}(x,z) \check{p}_{\lambda}(x+z) \hat{p}_{-\alpha}(x-z) \\
&\le (1+2z)^{\lambda} \hat{p}_{-\alpha}(x-z)\\
&\le 3^\lambda \check{z}^\lambda \hat{p}_{-\alpha}(x-z).
\end{split}
\]
Now we estimate $(\Delta_z \check{p}_{\lambda})(x)$ for $x\ge z\ge0$ and $\lambda\in[1,2]$. First, $(\Delta_z \check{p}_{\lambda})(x)=0$ if $x+z\le1$. Assuming $x+z>1$, three cases occur. If $x-z\ge1$, then $\check{p}_\lambda=p_\lambda$ on the relevant arguments, and the sublinearity of $r\mapsto r^{\lambda-1}$ gives
\[(\Delta_z \check{p}_{\lambda})(x)=(\Delta_z p_{\lambda})(x)=\lambda\int_0^z\bigl[(x+r)^{\lambda-1}-(x-r)^{\lambda-1}\bigr]\d r\le\lambda\int_0^z(2r)^{\lambda-1}\d r=2^{\lambda-1}\check{z}^{\lambda}.\]
If $x>1\ge x-z$, then $x\le2\check{z}$ and, using $\widecheck{(x-z)}=1$,
\[(\Delta_z \check{p}_{\lambda})(x)=(x+z)^{\lambda}-2x^{\lambda}+1\le\lambda\int_0^z(x+r)^{\lambda-1}\d r\le\lambda\int_0^z\bigl(x^{\lambda-1}+r^{\lambda-1}\bigr)\d r\le\bigl(\lambda2^{\lambda-1}+1\bigr)\check{z}^{\lambda}.\]
Finally, if $x-z<x\le1$, then $x\le2\check{z}$ and, arguing as in the previous case,
\[(\Delta_z \check{p}_{\lambda})(x)=(x+z)^{\lambda}-1\le(x+z)^{\lambda}-x^{\lambda}\le\lambda\int_0^z(x+r)^{\lambda-1}\d r\le\bigl(\lambda2^{\lambda-1}+1\bigr)\check{z}^{\lambda}.\]

Therefore, we have the inequality $(\Delta_z \check{p}_{\lambda})(x) \le (\lambda2^{\lambda-1}+1) \check{z}^{\lambda}$ for $x\ge z$, $\lambda\in[1,2]$ and using~\eqref{eq:generalbdd} together with~\eqref{eq:alpha bound K}, we have
\begin{equation*}\begin{split}
  \MoveEqLeft \iiint\d z \d x \d y\, (\Delta_z g)(x) K(x,y,z) e(x) d(y) \\
 &\le  \int \d z \int \d x \, \bigl((\lambda2^{\lambda-1}+1) \check{z}^{\lambda} \hat{x}^{-\alpha}+(1+2z)^{\lambda}\widehat{(x-z)}^{-\alpha}\bigr)
  |e(x)| \int \d y \,  K(x,y,z) d (y)\\
 & \le C_\lambda \int \d z\, \varphi(z) \hat{z}^{-\alpha} \check{z}^{\lambda} \int \d x \,  \hat{x}^{-\alpha} \check{x}^{\lambda}   |e(x)| \int \d y\, \hat{y}^{-\alpha} \check{y}^{\lambda} d (y)\\
&\quad + C_{\alpha,\lambda} \int \d z \, \varphi(z) \hat{z}^{-\alpha}\check{z}^\lambda  \int \d x \, \hat{ x}^{-\alpha}\check{x}^{\lambda} |e(x)|\int \d y\, \hat{y}^{-\alpha} \check{y}^{\lambda} d(y)
 \\
& = C_{\alpha,\lambda} \|\varphi\|_{-\alpha,\lambda} \Mom_{-\alpha,\lambda}(c,d) \Mom_{-\alpha, \lambda}(e),
\end{split}\end{equation*}
where we used $\hat{x}^{-\alpha} \le \hat{z}^{-\alpha}$ for $x \ge z$ in the last inequality. Since  $\Mom_{-\alpha,\lambda}(c,d)\le  \Mom_{-2\alpha,2\lambda}(c,d)$, we arrive at
\[
\frac{\d}{\d t}\Mom_{-\alpha,\lambda} (e(t)) \le  C_{\alpha,\lambda} \|\varphi\|_{-\alpha, \lambda} \Mom_{-2\alpha,2\lambda}(c,d)(t) \Mom_{-\alpha,\lambda}(e(t)).\]
So, Grönwall's lemma implies uniqueness of the solution.
\end{proof}

\section{Gelation}
In the following, we consider $\alpha=0$ which means the kernel $K$ is bounded at zero, but can still grow at infinity.
In this and the next chapter, the gelation will be shown via appropriate differential inequalities for the moments of the solution using the assumptions on the kernel. In comparison to the gelation results for exchange-driven growth~\cite{si2024existence}, the exchange of an arbitrary large mass encoded via the function~$\varphi$ has to be dealt with. Here, suitable integrability assumptions on certain moments of $\varphi$ allow us to adapt the arguments of~\cite{si2024existence}.
\begin{theorem}[Finite-time existence for quadratic growth]\label{thm:finite-existence:2}
Assume $K$ satisfies $K(x,y,z)\le \check{x}^2 \check{y}^2 \varphi(z)$, Equation  \eqref{ass:2.derivative,K} from Assumption~\ref{ass:existence} with $\alpha =0, \lambda =2$  and $\varphi \in Y^+_{0,2}$.
Then for any $0 \nequiv c_0 \in Y^+_{0,2}$, the weak solution to~\eqref{eq:CGEDG} on $[0,T_0)$, $T_0:=\bra*{ 2 \|\varphi\|_{0,2}\bigl( \Mom_0(c_0)+\Mom_2(c_0)\bigr)}^{-1}$ exists in $ Y_{0,2}^+$.
Moreover, it preserves the mass and the first moment on $[0,T_0)$.
\end{theorem}
\begin{remark}
	We note, since $c_0\in L^1(\bR_+)$, we get in particular $\Mom_2(c_0)>0$.
\end{remark}
\begin{proof}
	The argument follows closely along the lines for the existence of a lower growing kernel in Section~\ref{sec:existence}.
	Indeed, let $p_2(x)=x^2$. We use a Picard-Lindelöf argument analogous to Proposition~\ref{prop:truncated-sol-moment-preserve} to obtain the existence of solution for the truncated system, which preserves the mass and satisfies the moment bound
\[\begin{split}\frac{\d}{\d t} \int x^2 c^n_t(x) \d x &= \int (\Delta_z p_2)(x) \kappa_n[c^n](x,y,z) \d x\\
&\le 2  \iiint \d z \d x \d y\, z^2   \varphi(z) \check{x}^2 \check{y}^2 c^n_t(x) c^n_t(y)\\
&\le 2 \|\varphi\|_{0,2} \bigl(\Mom_0(c^n_0) + \Mom_2(c_t^n)\bigr)^2\\
&\le 2 \|\varphi\|_{0,2} \bigl(\Mom_0(c_0)+ \Mom_2(c^n_t)\bigr)^2.
\end{split}\]
The differential inequality implies
\[\Mom_2(c^n_t) \le \bra*{ \frac{1}{\Mom_0(c_0)+ \Mom_2(c_0)} - 2 \|\varphi\|_{0,2} t }^{-1} - \Mom_0(c_0) \] for $t <  \bra*{ 2 \|\varphi\|_{0,2} ( \Mom_0(c_0)+\Mom_2(c_0))}^{-1}=T_0$.
For $t\le T<T_0$ and setting $\alpha=0, \lambda=2$, we still have the moment bound as in Lemma \ref{lem:prop-lambda-moment} as well as the uniform integrability Proposition \ref{prop:UIbdd} and tightness from  Proposition \ref{prop:higher-moment-bdd}. Therefore, the compactness argument via Arzel\'{a}-Ascoli theorem, gives a subsequence limit $c$  and the arguments of Proposition~\ref{prop:limit-equation} and Proposition~\ref{prop:conservation-of-mass-1.moment}  show that it is a weak solution to~\eqref{eq:CGEDG} and conserves the mass and the first moment for times $t\in [0,T_0)$.
\end{proof}

\begin{lemma}[Propagation of moments]\label{lem:propofmom}
Assume $K$ satisfies   $K(x,y,z)\le \check{x}^2 \check{y}^2 \varphi(z)$, Equation  \eqref{ass:2.derivative,K} from Assumption~\ref{ass:existence} with $\alpha =0, \lambda =2$ and $\varphi \in Y^+_{0,2}$.
If $0 \nequiv c_0 \in Y^+_{0,r}$ for $r> 2$, then the weak solution to~\eqref{eq:CGEDG} constructed in Theorem~\ref{thm:finite-existence:2} on $[0,T_0)$, $T_0:=  \bra*{ 2 \|\varphi\|_{0,2}\bigl( \Mom_0(c_0)+\Mom_2(c_0)\bigr)}^{-1}$, satisfies
\begin{equation}\label{eq:highermombdd}
	\sup_{t\in [0,T]}\Mom_{r}(c_t) \le C(T),\quad \text{ for }  T< T_0 \,.
\end{equation}
\end{lemma}
\begin{remark}
	In the proof, we use from~Theorem \ref{thm:finite-existence:2} for $r=2$  the bound \begin{equation}\label{eq:uni2mombdd}
		\Mom_2(c_t) \le    \bra*{ \frac{1}{\Mom_0(c_0)+ \Mom_2(c_0)} - 2 \|\varphi\|_{0,2} t }^{-1} - \Mom_0(c_0)\qquad\forall t \in [0,T_0)  \,.
\end{equation} \end{remark}
\begin{proof}
For $n > 1$, we use $h^n_{r}(x)=\min(x^r, n^r)\in L^\infty(\bR_+)$ as test function of the weak solution and get
\[\int h^n_r(x) c_t(x) \d x - \int h^n_r(x) c_0(x) \d x= \int_0^t \d s \iiint \d z \d x \d y \, (\Delta_z h^n_r)(x) \kappa[c_s](x,y,z).\]
By applying the mean value theorem to the discrete Laplacian, using $r>2$ and the bound~\eqref{eq:uni2mombdd}, we estimate
\[\begin{split}  \int_0^t \d s& \iiint\d z \d x \d y \, (\Delta_z h^n_r)(x) \kappa[c_s](x,y,z)\\
&\le r (r-1)\int_0^t \d s \int_{0}^\infty \d z \, \varphi(z) \int_z^{n} \d x \int_z^\infty  \d y\,  (x+z)^{r-2} \check{x}^2 \check{y}^2  c_s(x) c_s(y)\\
&\le  r (r-1) 2^{r -2} \int_0^t \d s \int_{0}^\infty \d z\, \varphi(z) \int_z^{n} \d x\, x^{r-2} \check{x}^2c_s(x)  \int_z^\infty  \d y\, \check{y}^2   c_s(y)\\
&\le   r (r-1) 2^{r -2} \|\varphi\|_0 \int_0^t \d s \int_0^n \d x\, \check{x}^{r}c_s(x)  \int_0^\infty  \d y\, \check{y}^2   c_s(y)\\
&\le  r (r-1) 2^{r -2} \|\varphi\|_0 \int_0^t \d s\, \bigl(\Mom_0(c_0)+\Mom_2(c_s)\bigr)\biggl(\Mom_0(c_0)+\int_0^n \d x \, x^{r}c_s(x)\biggr) \\
&\le\frac{r (r-1) 2^{r -2} \|\varphi\|_0 }{(\Mom_0(c_0)+\Mom_2(c_0))^{-1}-2\|\varphi\|_{0,2}T} \int_0^t \d s \bra*{ \Mom_0(c_0)+\int_0^n \d x \, {x}^{r}c_s(x)}.
\end{split}\]
Hence, we arrive at the bound
\[
	\int_0^n \d x \, {x}^{r}c_t(x)\le \Mom_{r}(c_0) +\frac{r (r-1) 2^{r -2} \|\varphi\|_0 }{(\Mom_0(c_0)+\Mom_2(c_0))^{-1}-2\|\varphi\|_{0,2}T}  \bra*{ t\Mom_0(c_0)+\int_0^t \d s \int_0^n \d x \, {x}^{r}c_s(x)} \,.
\]
By Grönwall's lemma, we get with the constant $C_{\varphi,r, T,c_0}:=
\frac{r (r-1) 2^{r -2} \|\varphi\|_0 }{(\Mom_0(c_0)+\Mom_2(c_0))^{-1}-2\|\varphi\|_{0,2}T}>0$ the estimate
\[
	\int_0^n \d x\, {x}^{r}c_t(x)\le \bra[\big]{\Mom_r(c_0)+t \Mom_0(c_0) C_{\varphi,r, T,c_0}}\exp(C_{\varphi,r, T,c_0} t) \,.
 \]
By letting $n\to \infty$, we have
\[\Mom_r (c_t) \le \bra[\big]{\Mom_r(c_0)+T_0 \Mom_0(c_0) C_{\varphi,r, T,c_0}}\exp(C_{\varphi,r, T,c_0} T)\qquad \forall t\in[0,T]. \qedhere
\]
\end{proof}
\begin{proof}[Proof of Theorem \ref{thm:blow-up-time}]
For $n > 1$, let $h^n_{\mu}(x)=\min(x^\mu, n^\mu)$ be the truncated moment, which we use as a test function in the weak form~\eqref{eq:swCGEDG}, that is
\[\int h^n_\mu(x) c_t(x) \d x- \int h^n_\mu(x) c_0(x) \d x = \int_0^t \d s \iiint \d z \d x \d y \, (\Delta_z h^n_\mu)(x) \kappa[c_s](x,y,z) \,. \]
Since $h^n_\mu$ is non-decreasing and is a bounded truncation of a monomial with degree $\mu$, we have
\[\begin{split}
- z \mu x^{\mu -1 }\le (\Delta_z h^n_\mu)(x) \le h^n_\mu(x+z)-h^n_\mu(x) \le z \mu (x+z) ^{\mu -1} \le \mu 2^{\mu -1}  z x^{\mu -1}\end{split}.
\]
Therefore, for $t\in [0,T]$ with $T<T_0$, we arrive at the bound
\[\begin{split}\int_0^t \d s \iiint \d z \d x \d y & |(\Delta_z h^n_\mu)(x)| |\kappa[c_s](x,y,z)| \\
&\le
 \mu 2^{\mu -1}  \int_0^t \d s \int_0^\infty  \d z z  \varphi(z) \int_0^\infty \d x (1+x^{\mu +1}) c_s(x)\int_0^\infty  \d y (1+y^2)  c_s(y)\\
 &\le C_\mu T \|\varphi\|_{0,1} \sup_{s\in [0,T]}(\Mom_{0}(c_s)+\Mom_2(c_s)) \sup_{s\in [0,T]} ( \Mom_{0}(c_s)+\Mom_{1+\mu}(c_s)) < +\infty.
 \end{split}\]
The final bound is uniform in $n\in \bN$ because $1+\mu>2$ so that Lemma \ref{lem:propofmom} and~\eqref{eq:uni2mombdd} can be applied.
By dominated convergence, we can take the limit in $n\to \infty$ so that for $\mu \in (1,2]$, it holds
\[\int x^\mu c_t(x) \d x - \int x^\mu  c_0(x)  \d x= \int_0^t \d s \iiint \d z \d x \d y \, (\Delta_z p_\mu)(x) \kappa[c_s](x,y,z).\]
By the mean value theorem, $(\Delta_z p_\mu)(x) = \mu(\mu-1)(\theta_{x,z})^{\mu-2}z^2$, where $\theta_{x,z}\in [x-z,x+z]$ and using that the support of $K$ is on $x>z$, we obtain
\begin{equation}\begin{split}
\Mom_{\mu}(c_t) - \Mom_{\mu}(c_0)&\ge  \mu(\mu-1)2^{\mu-2} \int_0^t \d s \int_{0}^\infty \d z \, z^2\varphi_1(z)\int_z^\infty \d x \int_z^\infty \d y \, x^{\mu-2}\check{x}^{2}\check{y}^{\mu} c_s(x)c_s(y) \\
&\ge   \mu(\mu-1)2^{\mu-2}\int_0^t \d s \int_0^\infty\d z \, z^2\varphi_1(z) \bra*{\int_z^\infty \dx x \, x^{\mu} c_s(x) }^2 \,.
\end{split}\end{equation}
The squared integral can be further bounded from below as
\[\begin{split}\int_z^\infty \d x \, x^{\mu} c_s(x)  &=\int_0^\infty \d x \, x^{\mu} c_s(x) - \int_0^z \d x \, x^{\mu} c_s(x)\\
&\ge \Mom_\mu(c_s) - z^{\mu -1} \int_0^z \d x \, x c_s(x) \\
&\ge  \Mom_\mu(c_s) - z^{\mu -1}\Mom_1(c_s).
\end{split}\]
Since $\Mom_1(c_s)=\Mom_1(c_0)$, this reads $A(z):=\int_z^\infty x^\mu c_s(x)\,\d x\ge b(z):=\Mom_\mu(c_s)-z^{\mu-1}\Mom_1(c_0)$; note also $A\ge0$. Moreover, note that  $t\mapsto\Mom_\mu(c_t)$ is non-decreasing, because $(\Delta_z p_\mu)(x)\ge0$ for $\mu\ge1$, which makes the gain in the identity above non-negative.
In particular, $\Mom_\mu(c_s)\ge\Mom_\mu(c_0)$ for every $s\in[0,t]$.

We apply Jensen's inequality for the convex function $r\mapsto r^2$ with respect to the probability measure $\d\nu:=z^2\varphi_1(z)\,\d z/\Mom_2(\varphi_1)$ directly to $A$. Writing $K_0:=\frac{\Mom_{1+\mu}(\varphi_1)}{\Mom_2(\varphi_1)}\Mom_1(c_0)$, one has $\int b\,\d\nu=\Mom_\mu(c_s)-K_0$, so that
\[
\int_0^\infty z^2\varphi_1(z)\,A(z)^2\,\d z
=\Mom_2(\varphi_1)\int A^2\,\d\nu
\ge\Mom_2(\varphi_1)\biggl(\int A\,\d\nu\biggr)^2
\ge\Mom_2(\varphi_1)\bigl(\Mom_\mu(c_s)-K_0\bigr)^2,
\]
where the last step uses $\int A\,\d\nu\ge\int b\,\d\nu=\Mom_\mu(c_s)-K_0\ge0$, valid because $\Mom_\mu(c_s)\ge\Mom_\mu(c_0)>K_0$ by the assumption of the theorem. Substituting into the gain bound above yields
\begin{equation}\begin{split}
\Mom_{\mu}(c_t) - \Mom_{\mu}(c_0)
&\ge \mu(\mu-1)2^{\mu-2}\int_0^t \d s \int_0^\infty\d z\, z^2 \varphi_1(z)\bra*{\int_z^\infty x^\mu c_s(x)\,\d x}^2\\
&\ge  \mu(\mu-1)2^{\mu-2} \Mom_2(\varphi_1) \int_0^t \d s \, \bigl(\Mom_{\mu}(c_s) - K_0\bigr)^2 \,.
\end{split}
\end{equation}
Setting $\psi(s):=\Mom_\mu(c_s)-K_0$, we have $\psi(0)=\Mom_\mu(c_0)-K_0>0$ by assumption, and the previous display reads $\psi(t)-\psi(0)\ge\mu(\mu-1)2^{\mu-2}\Mom_2(\varphi_1)\int_0^t\psi(s)^2\,\d s$.
Since the integral equation $g(t) - g(0) = c \int_0^t \d s \, g(s)^2 $ with constant $c$  has  the solution
\[ g(t) = ( g(0)^{-1} - c t )^{-1},\]
we get by the comparison principle
\[\Mom_{\mu}(c_t) \ge \bra[\Bigg]{\frac{1}{\Mom_{\mu}(c_0)-\frac{\Mom_{1+\mu}(\varphi_1)}{\Mom_2(\varphi_1)}\Mom_1(c_0)} - \mu(\mu-1)2^{\mu-2}\Mom_2(\varphi_1)  t }^{-1}+\frac{\Mom_{1+\mu}(\varphi_1)}{\Mom_2(\varphi_1)}\Mom_1(c_0).\]
Hence the moment $\Mom_{\mu}(c_t)$ blows up at time \[ \bra*{\mu(\mu-1)2^{\mu-2}\Mom_2(\varphi_1)  \bra*{\Mom_{\mu}(c_0)-\frac{\Mom_{1+\mu}(\varphi_1)}{\Mom_2(\varphi_1)}\Mom_1(c_0)}}^{-1},\]
 concluding the proof.
\end{proof}
\section{Instantaneous gelation}
As in the finite-time gelation, the strategy here is to show the blow-up of moments. We observe that the second moment is non-decreasing.
\begin{lemma}\label{lem:non-dec-M2}
 Let $(c_t)_{t\ge 0}$ be a weak solution of~\eqref{eq:CGEDG} in $Y_{0,2}^+$ on $[0,T)$ and $0<T\le T_{gel}$.
 Suppose the solution satisfies the bound
 \[\int_0^t \d s \int \d z  \d x \d y \, z^2  \kappa[c_s](x,y,z) <+\infty.\]
 Then $\Mom_0(c(t))=\Mom_0(c(0))$ and $\Mom_1(c(t))= \Mom_1(c(0))$ for $t \in [0,T)$ and $t\mapsto \Mom_2(c(t))$ is non-decreasing on $[0,T)$.
\end{lemma}
\begin{proof}
First, for the weak solutions, we can test with the constant function $h=1$ and see that $\Mom_0(c(t))=\Mom_0(c(0))$. Now, since $(c_t)\in  Y_{0,2}^+$, we can extend the test function classes to functions of the form $f(x) = g(x)+m$ where $m$ is a constant and $g$ has uniformly bounded second derivative.
Consider the truncation with bounded first and second derivative, satisfying $f_n(x) = f(x) $ for $x\in [0,n]$, $f(x)=0$ for $x \ge 7n$  and $\|f''_n\|_{L^\infty}\le C_f$, $\|f'_n\|_{L^\infty} \le \|f'\|_{L^\infty}$ for $n$ large enough, where $C_f >0$ a constant depends only on $f\vert_{[0,1]},\|f''\|_{L^\infty}$.

Indeed, given $f$ with bounded first and second derivatives, then there exists a $C^2$ interpolation of $f \vert_{[0,n]}$, $f_1$ to the left so that $f'_1$ is monotone on $[-a,0]$ towards $f_1''(-a)=f_1'(-a)=0$.  By a constant shift $\tilde f = f - f_1(-a)$, we have $\tilde f(-a) = 0$. And we will drop $\tilde{}$ from now on. Then by a 180-degree rotation of the graph of $f_1$ at $(n,f(n))$, we extend $f_1$ to $[-a,2n+a]$. We call this extension $f_2$. Further, by a reflection of the graph of $f_2$ along the line $x= 2n+a$, we obtain $f_3$ by extending $f_2$ to $[-a,4n+3a]$. Note that by construction this extension $f_3(4n+3a)=f_3'(4n+3a)=f_3''(4n+3a)= 0$ so that we define $f_4\in C^2$ by extending $f_3$ to zero on $\bR_+$ for $x \ge 4n+3a$. By choosing $n$ larger than $a$, this gives a desired interpolation of $f$ on~$[0,7n]$.

With these preliminary considerations, we use $f_n$ as a test function in the weak form
\[
\int_0^\infty f_n (x) [c_t( x) - c_0( x) ]\d x  = \int_0^t \d s \iiint  \d z \d x \d y  \, (\Delta_z f_n)(x) \kappa[c_s](x,y,z)\, .\]
Since $|f_n(x)| \le  \|f'\|_{L^\infty} x + |f(0)|$, we get on the one hand
\[\int_0^\infty |f_n (x)|c_t( x) \d x\le  \int_0^\infty (\|f'\|_{L^\infty} x + |f(0)|)  c_t(x) \d x <+\infty \]
and since by construction $|(\Delta_z f_n)|\le  C_f z^2 $, we get on the other hand
\[\begin{split} \int_0^t \d s \iiint  \d z \d x \d y  \, |(\Delta_z f_n)(x) | \kappa[c_s](x,y,z) \le C_f  \int_0^t \d s \iiint  \d z  \d x \d y\, z^2  \kappa[c_s](x,y,z)
<+\infty. \end{split}\]
Hence, by dominated convergence, functions with bounded first and second derivatives are admissible in the weak form~\eqref{eq:swCGEDG}.
In particular, we obtain $\Mom_1(c(t))=\Mom_1(c(0))$ by taking $h(x)=x$ and for $h(x) = x^2$ we get
\[
\Mom_2(c(t))-\Mom_2(c(0)) = 2  \int_0^t \d s \iiint  \d z \d x \d y  \, z^2 \kappa[c_s](x,y,z)\ge 0,
\]
which implies $t\mapsto\Mom_2(c(t))$ is non-decreasing.
\end{proof}
The estimate for the instantaneous gelation is based on the representation of moments via the tail distribution and derives the evolution of those in the next lemma.
\begin{lemma}[Evolution of weighted tail distributions]\label{lem:tail}
Let $g$ be an admissible test function for the weak form~\eqref{eq:swCGEDG}, which is locally bounded.
Then the following representation of the tail distribution for $\ell \ge 0$ holds
\begin{align}\label{eq: test-function-tail}
	\int_\ell^\infty \d x \, g(x)\bigl( c_t (x) - c_0(x) \bigr)
	&= \int_0^t \d s \iiint \d z \d x \d y \, \Bigl[ \bigl(\Delta_z g\bigr)(x) \1_{[\ell,\infty)}(x)   \kappa[c_s](x,y,z) \\
&\quad + g(x)\bra*{\1_{[\ell,\ell+z]}(x)\kappa[c_s](x-z,y,z) -\1_{[\ell-z,\ell]}(x) \kappa[c_s](x+z,y,z)} \Bigr]. \nonumber
\end{align}
\end{lemma}
\begin{proof}
We use the definition of weak solutions~\eqref{eq:swCGEDG} with test functions $g$ and $g\1_{[0,\ell]}\in L^\infty(\bR_+)$ and the chain rule for the discrete Laplacian~\eqref{eq:id:discretelaplace} to get
\[\begin{split}
	\MoveEqLeft\int_\ell^\infty \d x \, g(x) ( c_t (x) - c_0(x) ) = \int_0^\infty \d x \,g(x)( c_t (x) - c_0(x) ) - \int_0^\ell \d x \, g(x)( c_t (x) - c_0(x) )  \\
&\begin{multlined}
	= \int_0^t \d s \iiint \d z \d x \d y \Bigl[ (\Delta_z g)(x) \1_{[\ell,\infty)}(x) \\
 +  g(x+z) \1_{[\ell,\ell+z]}(x+z) - g(x-z) \1_{[\ell-z,\ell]}(x-z)\Bigr] \kappa[c_s](x,y,z)
\end{multlined}\\
&\begin{multlined}
	=\int_0^t \d s \iiint \d z \d x \d y \Bigl[ (\Delta_z g)(x) \1_{[\ell,\infty)}(x) \kappa[c_s](x,y,z) \\
	+ g(x)\bra*{\1_{[\ell,\ell+z]}(x)\kappa[c_s](x-z,y,z) -\1_{[\ell-z,\ell]}(x) \kappa[c_s](x+z,y,z)} \Bigr] \,,
\end{multlined}
\end{split} \]
which is the claimed identity
\end{proof}
\begin{lemma}\label{lem:propagation of moment for lower bdd kernel}
Assume $K$ satisfies $K(x,y,z)\ge (x^{\beta}+y^{\beta}) \varphi_1(z)\1_{(z,+\infty)^2}(x,y)$ with $\varphi_1\in Y_{0,2}^+$ satisfying~\eqref{eq:Hcond} and $\beta >2$.
For $c_0\in  Y_{0,2}^+$ with $\Mom_0(c_0)>0$ let $(c_t)_{t\ge 0}$ be any weak solution of~\eqref{eq:CGEDG} on $[0,T)$ for $0<T\le T_{gel}$ in $ Y_{0,2}^+$.
Then, the solutions satisfies for any $p \ge 1 $, $\sup_{s \in [0,t]}\Mom_p(c_s)<+\infty$  for all $t\in [0,T)$.
\end{lemma}
\begin{proof}
Let $\ell>0$ be an arbitrary level. We apply Lemma~\ref{lem:tail} with the test function $p_2(x)-\ell^2$ and note that $(\Delta_z (p_2 - \ell^2)) (x) = 2z^2$. In this way, we get for $0\leq r < t \leq T$ that
\begin{equation*}\begin{split}
\MoveEqLeft \int_\ell^\infty \d x \, (x^2- \ell^2)(c_t(x)-c_{r}(x)) = \int_{r}^t  \d s \iiint \d z  \d x \d y \, 2 z^2 \1_{[\ell,\infty)}(x) \kappa[c_s](x,y,z) \\
&\quad + \int_{r}^t  \d s \iiint \d z \d x \d y  \, (x^2 - \ell^2) \bra*{\1_{[\ell,\ell+z]}(x)\kappa[c_s](x-z,y,z) -\1_{[\ell-z,\ell]}(x) \kappa[c_s](x+z,y,z)} \\
&\ge \int_{r}^t  \d s \iiint \d z  \d x \d y  \, 2 z^2 \1_{[\ell,\infty)}(x) \kappa[c_s](x,y,z).
\end{split}\end{equation*}
In the last inequality, we used that $x^2 - \ell^2 \ge 0$ on $ x\in [\ell,\ell+z]$ and $-(x^2 -\ell^2) \ge 0$ on  $x\in [\ell-z,\ell]$.
Next, by abbreviating $W_r:=\int_\ell^\infty (x^2-\ell^2)\, c_r(x)\,\d x\ge 0$ and using the assumption $K(x,y,z)\ge \varphi_1(z) x^{\beta}$ on its support $\{z\le x\wedge y\}$, we obtain for $0\le r\le t<T$ the bound
\[
W_t-W_r \ge 2\int_r^t \d s \int_\ell^\infty \d x\, x^{\beta} c_s(x)\int_0^{x}\d z\, z^2\varphi_1(z)\int_z^\infty \d y\, c_s(y).
\]%
By Definition~\ref{def:weaksolution}, the function $c\in C([0,T];L^1)$ is uniformly integrable on $[0,T]$ and we find $z_\ast\in(0,\tfrac12]$ with $\inf_{s\in[0,T]}\int_{z_\ast}^\infty c_s(y)\,\d y\ge\tfrac12\Mom_0(c_0)$. Assuming $\ell\ge z_\ast$ and keeping only the contributions of $z\le z_\ast$ and $y\ge z_\ast$, together with $x^{\beta}\ge \ell^{\beta-2}x^2\ge \ell^{\beta-2}(x^2-\ell^2)$ on $[\ell,\infty)$, this gives, since $\int_0^{z_\ast}z^2\varphi_1(z)\,\d z>0$ by~\eqref{eq:Hcond} (as $z^2>0$ for $z>0$), the Gr\"onwall inequality
\[
W_t-W_r\ge a_\ell\int_r^t W_s\,\d s\,,\qquad a_\ell:=\Mom_0(c_0)\Bigl(\int_0^{z_\ast}z^2\varphi_1(z)\,\d z\Bigr)\ell^{\beta-2}>0\,,
\]
applied to the non-negative map $t\mapsto W_t$ on $[r,T]$ (with $r$ fixed), which gives $W_t\ge W_r \exp(a_\ell(t-r))$. Since $W_t\le \Mom_2(c_t)\le \Mom_2(c_T)$ by the non-decreasing property of the second moment from Lemma~\ref{lem:non-dec-M2}, we conclude that
\begin{equation}\label{eq:tail-W-bound}
\int_\ell^\infty (x^2-\ell^2)\, c_r(x)\,\d x \le \Mom_2(c_T)\, \exp\bigl(-a_\ell(t-r)\bigr)\qquad\text{for every $\ell\ge z_\ast$.}
\end{equation}%
Since $y^2\le\frac43(y^2-\ell^2)$ for $y\ge 2\ell$ while $y^2-\ell^2\ge 0$ for $y\ge \ell$, the bound~\eqref{eq:tail-W-bound} yields
\begin{equation}\label{eq:tail-2-moment-bdd}
\int_{2\ell}^\infty y^2 c_r(y)\,\d y \le \frac43 \int_\ell^\infty (y^2-\ell^2)\, c_r(y)\,\d y \le \frac43 \Mom_2(c_T)\exp\bigl(-a_\ell(t-r)\bigr)
\qquad\text{for every $\ell\ge z_\ast$.}
\end{equation}
In the case $p\ge 2$, we decompose $[1,\infty)$ dyadically and apply~\eqref{eq:tail-2-moment-bdd} at the levels $\ell=2^{k-1}$ to obtain
\begin{align}\label{eq:p-moment-uniform-bdd}
\int_1^\infty x^p c_r(x)\,\d x &= \sum_{k=0}^\infty \int_{2^k}^{2^{k+1}} x^p c_r(x)\,\d x \le \sum_{k=0}^\infty 2^{(k+1)(p-2)} \int_{2^k}^\infty x^2 c_r(x)\,\d x \\
&\le \frac43 \Mom_2(c_T)\sum_{k=0}^\infty 2^{(k+1)(p-2)} \exp\bigl(-a_{2^{k-1}}(t-r)\bigr).\nonumber
\end{align}
Since $\beta>2$, the exponent $a_{2^{k-1}}=\Mom_0(c_0)\bigl(\int_0^{z_\ast}z^2\varphi_1(z)\,\d z\bigr)\,2^{(k-1)(\beta-2)}$ grows geometrically in $k$ and hence dominates the prefactor $2^{(k+1)(p-2)}$, so the series in~\eqref{eq:p-moment-uniform-bdd} converges. Note that by the same argument, for every $t <T$, we can find a uniform bound for $0 \le s\le t < T $.  Together with $\int_0^1 x^p c_r(x)\,\d x\le \Mom_0(c_0)<+\infty$, this gives for $0\le t < T$,  $\sup_{s \in [0,t]} \Mom_p(c_s)<+\infty$ for $p\ge 2$; for $1\le p<2$ the H\"older inequality gives $\Mom_p(c_r)\le \Mom_0(c_r)^{1-p/2}\Mom_2(c_r)^{p/2}<+\infty$.
 Therefore,  $\sup_{s \in [0,t]}\Mom_p(c_s)<+\infty$   for $p\ge 1$, for $0 \le t <T$.
\end{proof}
The next estimate is the continuous counterpart of the discrete second--difference bound $(j+1)^n-2j^n+(j-1)^n\ge n(n-1)j^{n-2}$ of~\cite[Proof of Theorem~2.9]{si2024existence}, with the sharp constant $n(n-1)$. The sharpness, i.e.\ the prefactor $1$, is essential, since replacing it by any factor that decreases with $n$ would destroy the vanishing of the blow-up time $T_n$ as $n\to\infty$.
\begin{lemma}\label{lem:secdiff}
For every $n\ge2$ and all $x\ge z\ge0$, writing $P_n(x):=(1+x)^n$,
\[(\Delta_z P_n)(x)\ge n(n-1)(1+x)^{n-2}z^2.\]
\end{lemma}
\begin{proof}
Let us define $w:=1+x\ge1$ and $f(w):=w^n$, so that $(\Delta_z P_n)(x)=f(w+z)-2f(w)+f(w-z)$ with $w-z=1+x-z\ge1$. Since $f(w+z)-f(w)=\int_0^z f'(w+s)\,\d s$ and $f(w)-f(w-z)=\int_{-z}^0 f'(w+s)\,\d s$, we have the identity
\[(\Delta_z P_n)(x)=\int_0^z\bigl(f'(w+s)-f'(w-s)\bigr)\,\d s=\int_{-z}^z(z-|u|)\,f''(w+u)\,\d u.\]
The weight $(z-|u|)\,\d u$ is non-negative and symmetric on $[-z,z]$, with total mass $\int_{-z}^z(z-|u|)\,\d u=z^2$. As $f''(w+u)=n(n-1)(w+u)^{n-2}$ and $u\mapsto(w+u)^{n-2}$ is convex (because $n-2\ge0$ and $w-z>0$), Jensen's inequality yields $\int_{-z}^z(z-|u|)(w+u)^{n-2}\,\d u\ge z^2 w^{n-2}$, which is the claim.
\end{proof}
\begin{proof}[Proof of Theorem \ref{thm:instant-gel}] We argue by contradiction, adapting the discrete moment-blow-up argument of~\cite[Theorem~2.9]{si2024existence}. Suppose a weak solution $c$ to~\eqref{eq:CGEDG} in $Y_{0,2}^+$ exists on $[0,T)$ with $0<T\le T_{gel}$, and fix $t\in(0,T)$. On $[0,t]$ the mass and first moment are conserved, $\Mom_0(c_s)=\Mom_0(c_0)>0$ and $\Mom_1(c_s)=\Mom_1(c_0)$ by Lemma~\ref{lem:non-dec-M2} and all moments are finite with $\sup_{s\in[0,t]}\Mom_p(c_s)<\infty$ for every $p\ge1$ by Lemma~\ref{lem:propagation of moment for lower bdd kernel}.
Moreover, as a weak solution $\{c_s\}_{s\in[0,t]}$ is uniformly integrable, being relatively compact in $L^1(\bR_+)$.
Let $p_m(x) =x^{m}$, $m\in\bN$, $m \ge 2$.
We consider the test function $p_m \1_{[0,n]}$ in the weak form
\[\int_0^n p_m(x) [c_t( x) - c_0( x) ]\d x  = \int_0^t \d s \iiint \d z \d x \d y  \, (\Delta_z p_m)(x)\1_{[0,n]}(x) \kappa[c_s](x,y,z) + \text{ bdry.~terms}.
\]
The boundary terms are from the product identity~\eqref{eq:id:discretelaplace} for the discrete Laplacian, see also the proof of Theorem \ref{thm:uniquness}. On the support of $K$ we have $z\le x$, hence by the mean value theorem $(\Delta_z p_m)(x)=m(m-1)\theta_{x,z}^{m-2}z^2\le m(m-1)2^{m-2}z^2 x^{m-2}$. Together with $K(x,y,z)\le\varphi(z)(\check{x}^k+\check{y}^k)$, the integrand $(\Delta_z p_m)(x)\kappa[c_s](x,y,z)$ is dominated by
\[m(m-1)2^{m-2}\, z^2\varphi(z)\, \check{x}^{m-2}(\check{x}^k+\check{y}^k)\, c_s(x)c_s(y),\]
whose integral is a finite combination of moments $\Mom_j(c_s)$ with $0\le j\le m+k-2$. These are uniformly bounded for $s\le t<T$ by~\eqref{eq:p-moment-uniform-bdd} of Lemma \ref{lem:propagation of moment for lower bdd kernel}, so dominated convergence applies. Therefore for each $t<T$, we have the convergence to
\[
\int_0^t \d s \iiint  \d z \d x \d y  \, (\Delta_z p_m)(x) \kappa[c_s](x,y,z).\]
Similar to the argument of \eqref{eq:domiantedbound}, the upper bound of  $K(x,y,z) \le\varphi(z)( \check{x}^{k} + \check{y}^k )$ implies that the boundary terms with $p_{m}(x+z)\1_{[n-z,n]}(x)$ and $p_{m}(x-z)\1_{[n,n+z]}(x)$ at time $s$ of the truncated function $\1_{[0,n]}(x)p_{m}(x) $  can be bounded by
\[C_{m} \int \d z \, \varphi(z) \int_{n/2}^n  \d x \d y \, x^{m} (x^{k}+ y^{k}) c_s(x) c_s(y)\] for some constant $C_m >0$, which vanishes as $n\to \infty$ if $c_s \in Y_{m + k}^+ $ for all $s<t<T$. The latter is again guaranteed by Lemma \ref{lem:propagation of moment for lower bdd kernel}.  Then, via a dominated convergence argument on the time integral, we see that the boundary terms vanish. Therefore, $p_m$ is an admissible test function of the weak solution.

By linearity, every $(1+x)^n$, $n\in\bN$, is an admissible, locally bounded test function as well, so Lemma~\ref{lem:tail} applies to it.
By the uniform integrability there is $z_\ast\in(0,1)$ with
\begin{equation}\label{eq:UIrecip}
\delta:=\tfrac12\Mom_0(c_0)\le\int_{z_\ast}^\infty c_s(y)\,\d y\qquad\text{for all }s\in[0,t],
\end{equation}
and, by~\eqref{eq:Hcond} (as $z^2>0$ for $z>0$), $G:=\int_0^{z_\ast}z^2\varphi_1(z)\,\d z>0$; the constants $z_\ast,\delta,G$ are independent of $n$. For $p\ge0$,
we define the tail moment $\mathcal{Y}_p(s):=\int_{z_\ast}^\infty(1+x)^p c_s(x)\,\d x$, so that $\mathcal{Y}_p(s)\ge(1+z_\ast)^p\delta$ by~\eqref{eq:UIrecip}.

Next, we fix $n\ge2$ and apply Lemma~\ref{lem:tail} with the test function $(1+x)^n$ and level $z_\ast$ for $0\le\sigma\le\tau\le t$ to arrive at
\[
	\mathcal{Y}_n(\tau)-\mathcal{Y}_n(\sigma)\ge\int_\sigma^\tau\bigl(\mathcal D_n (s)-\mathcal B_n(s)\bigr)\d s \,,
\]
where the non-negative contribution carrying $\1_{[z_\ast,z_\ast+z]}(x)$ in \eqref{eq: test-function-tail} has been dropped and
\begin{align*}
	\mathcal D_n(s) &:= \iiint \d z \d x \d y \bigl(\Delta_z(1+\cdot)^n\bigr)(x)\,\1_{[z_\ast,\infty)}(x)\,\kappa[c_s](x,y,z) \,, \\
	\mathcal B_n(s) &:= \iiint \d z \d x \d y (1+x)^n\,\1_{[z_\ast-z,z_\ast]}(x)\,\kappa[c_s](x+z,y,z) \,.
\end{align*}
On the support $\{z\le x\wedge y\}$ of $K$ one has $\check x\ge\tfrac{1+x}{2}$, hence $K(x,y,z)\ge\varphi_1(z)\check x^\beta\ge2^{-\beta}\varphi_1(z)(1+x)^\beta$. Applying together with Lemma~\ref{lem:secdiff} and $\int_0^{x\wedge y}z^2\varphi_1(z) \d z \ge G$ for $x,y\ge z_\ast$, we get
\begin{equation}\label{eq:gainbd}
\mathcal D_n(s)\ge2^{-\beta}n(n-1)G\int_{z_\ast}^\infty(1+x)^{n+\beta-2}c_s(x)\,\d x\int_{z_\ast}^\infty c_s(y)\,\d y\ge2^{-\beta}n(n-1)G\delta\,\mathcal{Y}_{n+\beta-2}(s).
\end{equation}
We define $\Lambda:=\tfrac{\beta-2}{n-1}>0$ and, on $[z_\ast,\infty)$, the probability measure $\d\mu_s(x):=\mathcal{Y}_1(s)^{-1}(1+x)c_s(x)\,\d x$  to get the identity
\[
\int (1+x)^{n-1} \,\d\mu_s(x)=\frac{1}{\mathcal{Y}_1(s)}\int_{z_\ast}^\infty(1+x)^n c_s(x)\,\d x=\frac{\mathcal{Y}_n(s)}{\mathcal{Y}_1(s)} \,.
\]
We now apply the Jensen's inequality $\int\Psi(X)\,\d\mu_s\ge\Psi\bigl(\int X\,\d\mu_s\bigr)$ with the convex function $\Psi(r):=r^{1+\Lambda}$ and for $X(x):=(1+x)^{n-1}$  and note that the identity $1+(n-1)(1+\Lambda)=n+\beta-2$ yields
\[
\int \Psi(X)\,\d\mu_s=\frac{1}{\mathcal{Y}_1(s)}\int_{z_\ast}^\infty(1+x)^{n+\beta-2}c_s(x)\,\d x=\frac{\mathcal{Y}_{n+\beta-2}(s)}{\mathcal{Y}_1(s)} \,.
\]
Therefore, the chain of the two inequalities becomes $\mathcal{Y}_{n+\beta-2}(s)/\mathcal{Y}_1(s)\ge\bigl(\mathcal{Y}_n(s)/\mathcal{Y}_1(s)\bigr)^{1+\Lambda}$, which after multiplying by $\mathcal{Y}_1(s)$ and using $\mathcal{Y}_1(s)\le m_{01}:=\Mom_0(c_0)+\Mom_1(c_0)$ becomes the desired bound
\begin{equation}\label{eq:jensentail}
\mathcal{Y}_{n+\beta-2}(s)\ge \mathcal{Y}_1(s)^{-\Lambda}\mathcal{Y}_n(s)^{1+\Lambda}\ge m_{01}^{-\Lambda}\mathcal{Y}_n(s)^{1+\Lambda}.
\end{equation}
On the indicator $\1_{[z_\ast-z,z_\ast]}(x)$ of $\mathcal B_n$, we have $x\le z_\ast<1$, so $(1+x)^n\le(1+z_\ast)^n$. Substituting $w=x+z$ turns this indicator into $\1_{[z_\ast,z_\ast+z]}(w)=\1_{\{w\ge z_\ast\}}\,\1_{\{z\ge w-z_\ast\}}$; bounding $\1_{\{w\ge z_\ast\}}\le1$ (which extends the $w$-integral to $[0,\infty)$ and only enlarges the bound) and using $K(w,y,z)\le\varphi(z)(\check w^k+\check y^k)$ together with $\int_0^\infty\varphi(z)\1_{\{z\ge w-z_\ast\}}\,\d z\le\|\varphi\|_0$, we get
\begin{equation}\label{eq:Bbd}
\mathcal B_n(s)\le(1+z_\ast)^n B
\quad\text{with}\quad
B:=2\|\varphi\|_0\Mom_0(c_0)\Bigl(\Mom_0(c_0)+\sup_{s\in [0,t]}\Mom_k(c_s)\Bigr)<\infty,
\end{equation}
which is finite by Lemma~\ref{lem:propagation of moment for lower bdd kernel} uniformly in $n\geq 1$. By~\eqref{eq:gainbd}, \eqref{eq:jensentail} and $\mathcal{Y}_n(s)\ge(1+z_\ast)^n\delta$, we conclude
\[
	\frac{\mathcal D_n(s)}{\mathcal B_n(s)}\ge\frac{2^{-\beta}n(n-1)G\delta^{2+\Lambda}m_{01}^{-\Lambda}}{B}\,(1+z_\ast)^{n\Lambda}\xrightarrow[n\to\infty]{}\infty \,,
\]
because $n\Lambda=\tfrac{n}{n-1}(\beta-2)\ge\beta-2>0$ and the prefactor grows like $n^2$. Fix $N_0$ so that $\mathcal D_n(s)\ge2\mathcal B_n(s)$ for all $n\ge N_0$ and $s\in[0,t]$; then $\mathcal D_n-\mathcal B_n\ge\tfrac12\mathcal D_n$ and with~\eqref{eq:gainbd}, \eqref{eq:jensentail}, we obtain
\begin{equation}\label{eq:closedYn}
\mathcal{Y}_n(\tau)-\mathcal{Y}_n(\sigma)\ge A_n\int_\sigma^\tau \mathcal{Y}_n(s)^{1+\Lambda}\,\d s,\qquad A_n:=2^{-\beta-1}n(n-1)G\delta\,m_{01}^{-\Lambda}>0.
\end{equation}
Fix $n\ge N_0$. By~\eqref{eq:closedYn} with $\sigma=0$ and $\mathcal{Y}_n(0)\ge(1+z_\ast)^n\delta>0$, comparison with the ordinary differential equation $g'=A_ng^{1+\Lambda}$, $g(0)=\mathcal{Y}_n(0)$, whose solution blows up at $\mathcal{Y}_n(0)^{-\Lambda}/(A_n\Lambda)$, shows that $\mathcal{Y}_n$ becomes infinite no later than
\[T_n:=\frac{\mathcal{Y}_n(0)^{-\Lambda}}{A_n\Lambda}\le\frac{2^{\beta+1}m_{01}^{\Lambda}\delta^{-1-\Lambda}}{(\beta-2)G}\,\frac{(1+z_\ast)^{-n\Lambda}}{n}\xrightarrow[n\to\infty]{}0,\]
since $A_n\Lambda=2^{-\beta-1}(\beta-2)G\delta\,m_{01}^{-\Lambda}n$ and $(1+z_\ast)^{-n\Lambda}\to(1+z_\ast)^{-(\beta-2)}$. Choosing $n\ge N_0$ with $T_n<t$ gives $\mathcal{Y}_n(\tau)=+\infty$ for some $\tau<t<T$. Now as $(1+x)^n\le 2^n x^n$ for $x\ge1$, we have \[\mathcal{Y}_n(\tau) \le  2^n \int_{1}^{\infty} x^n c_\tau(x) \d x ,\] this forces $\Mom_n(c_\tau)=+\infty$, contradicting the finiteness of all moments from Lemma~\ref{lem:propagation of moment for lower bdd kernel}. Hence no such solution exists, that is, $T_{gel}=0$.
\end{proof}

\appendix
\section{Reformulation of Assumption~\ref{ass:existence} from Remark \ref{rmk:example} }
\begin{proposition}\label{prop: example} Suppose the assumption \eqref{eq:generalbdd} holds. If $K$ satisfies \eqref{ass:truncation}, then \eqref{eq:alpha bound K}  holds.
\end{proposition}
\begin{proof} Given $x \ge z \ge 0$, $y \ge 0$, for $x-z >1$ we have  $x > 1$, so that by Assumption  \ref{ass:global}
\[K(x,y,z)\le \hat{y}^{-\alpha} x^{\lambda} \check{y}^{\lambda}\varphi(z).\]
For $x-z < 1$ and $1-\frac{z}{x}\le \Omega$, by Equation \eqref{ass:truncation}, we have
\[(x-z)^{-\alpha} K(x,y,z) \le   x^{-\alpha} \hat{x}^{-\alpha}   \hat{y}^{-\alpha} \check{x}^{\lambda}\check{y}^{\lambda} \varphi(z) \le   \hat{x}^{-2\alpha}\hat{y}^{-\alpha}   \check{x}^{\lambda}\check{y}^{\lambda} \varphi(z) . \]
For $x-z <1$ and $1-\frac{z}{x} > \Omega$, we have
\[(x-z)^{-\alpha} K(x,y,z) \le \biggl(1-\frac{z}{x}\biggr)^{-\alpha} \hat{x}^{-2 \alpha}\hat{y}^{-\alpha} \check{x}^\lambda  \check{ y}^\lambda  \varphi(z)\le \Omega^{-\alpha} \hat{x}^{-2 \alpha}\hat{y}^{-\alpha} \check{x}^\lambda  \check{ y}^\lambda  \varphi(z). \]
Therefore, in all cases we have the estimate  \eqref{eq:alpha bound K} with $C_\alpha = \Omega ^{-\alpha}$.
\end{proof}

\bibliographystyle{abbrv}
\bibliography{bib.bib}
\end{document}